\documentclass{svjour3}
\usepackage{amssymb}

\usepackage{epic}
\usepackage{amsmath}
\usepackage{amsfonts}
\usepackage{pstricks}
\usepackage{float}
\usepackage{subfig}
\usepackage{graphicx}
\usepackage{multirow}
\usepackage{color}
\usepackage[subfigure]{tocloft}
\usepackage{subfig}
\usepackage[utf8]{inputenc}
\usepackage{soul}
\usepackage{hyperref}
\hypersetup{colorlinks=true,linkcolor=blue,urlcolor=blue}
\def\argmin{\mathop{\text{argmin}}}
\def\mig{\frac{1}{2}}

\def\bigO{\mathcal{O}}
\journalname{Journal of Scientific Computing}

\begin{document}
\title{High order boundary extrapolation technique for finite difference methods
  on complex domains with Cartesian meshes}
\titlerunning{High order boundary extrapolation}
\author{A.~Baeza \and P.~Mulet \and D.~Zorío}
\date{}

\institute{
Departament de Matemàtica Aplicada,  Universitat de
  València    (Spain); emails: antonio.baeza@uv.es, mulet@uv.es, david.zorio@uv.es. This research was partially supported by Spanish MINECO grants
   MTM2011-22741 and MTM2014-54388.
}

\leavevmode\thispagestyle{empty}

\noindent This version of the article has been accepted for publication, after a peer-review process, and is subject to Springer Nature’s AM terms of use, but is not the Version of Record and does not reflect post-acceptance improvements, or any corrections. The Version of Record is available online at: \url{https://doi.org/10.1007/s10915-015-0043-2}

\newpage

\maketitle

\begin{abstract}
The application of suitable numerical boundary conditions for
hyperbolic conservation laws on domains with complex geometry has
become a problem with certain difficulty that has been tackled in
different ways according to the nature of the numerical methods and
mesh type.
In this paper we present a technique for the extrapolation of
information from the interior of the computational domain to
ghost cells designed for structured Cartesian meshes (which, as opposed to non-structured
meshes, cannot be adapted to the morphology of the domain boundary).

This technique is based on the application of Lagrange
interpolation with a filter for the detection of
discontinuities that permits a data dependent extrapolation, with
higher order at smooth regions and essentially non oscillatory
properties near discontinuities.

\end{abstract}

\section{Introduction}
Hyperbolic conservation laws and related equations have been the focus
of many research lines in the past four decades. Since no analytic
solution is known for many of these equations, different techniques
have been developed in order to tackle these problems from a numerical
point of view, with methodologies that have evolved and improved
along these years.

Some of these methods, mainly used for academic purposes, employ
Cartesian meshes on rectangular domains and numerical boundary
conditions that are based on low order extrapolation, generally first
order.
Due to the advantages of implementation and efficiency inherent to
Cartesian meshes, in our case we will focus on extending
the techniques mentioned above to domains with complex geometry (i.e.,
at least not necessarily rectangular). We will also present  a new
technique for the extrapolation of interior information to ghost
cells (cells  outside the domain, but  within the stencils of interior
points) making use of boundary conditions (if available) and interior
data near a given ghost cell. This procedure is able to detect abrupt data
changes.

Some authors have approached this problem from different perspectives.
 In \cite{Sjogreen} the authors develop a technique based on Lagrange interpolation
 with a limiter which is restricted to second order methods and
 a single ghost cell.
Also related to our approach are the works of Shu and collaborators
\cite{TanShu,TanWangShu}
where the equation to be solved is used to extrapolate derivative values of the numerical solution
to the boundary points where inflow conditions are prescribed and then approximate ghost values
by a Taylor expansion. For outflow boundaries
an extrapolation technique based on the WENO method is used, achieving high order when the data is smooth in both cases.
The drawbacks of this approach are that it is problem-dependent
(see \cite{Huang,Xiong} for a similar methodology applied to other equations), that it requires a different treatment of different
types of boundary and its relatively high computational cost.

Our approach can be understood as an extension of \cite{Sjogreen} in the sense that it is based on Lagrange extrapolation
with filters, but without imposing limitations on the order of the method
or the number of ghost cells. Further, albeit the description is made for hyperbolic conservation laws,
the procedure is agnostic about the equation and can be applied to other hyperbolic problems.
Finally, the methodology is the same for inflow and outflow boundaries, just by considering the boundary
node as an interpolation node in the case of inflow data.

The organization of the paper is the following: In section \ref{seq}
we present the equations and the numerical methods that we
consider in this paper. The details of the procedure for meshing complex domains
with Cartesian meshes are explained in section \ref{sml}. In section
\ref{sep} we expound how we perform extrapolations with the method for
the detection of singularities. Some numerical results that are
obtained with this methodology are presented in section \ref{srn},
with some simple tests in 1D to illustrate the correct behavior of the
proposed techniques and some more complex ones in 2D. Finally, some
conclusions are drawn in section \ref{scn}.

\section{Numerical schemes}\label{seq}
The equations that will be considered throughout this paper are
hyperbolic systems of $m$ two-dimensional  conservation laws
\begin{equation}\label{eq:hypsis}
u_t+f(u)_x+g(u)_y=0,\quad u=u(x,y,t),
\end{equation}
defined on an open and bounded spatial domain
$\Omega\subseteq\mathbb{R}^2$, with
Lipschitz boundary $\partial\Omega$ given by a finite union of piece-wise smooth
curves, $u:\Omega\times\mathbb{R}^+\rightarrow
\mathbb{R}^m$, and fluxes $f,g:\mathbb{R}^m\rightarrow\mathbb{R}^m$.
  These equations are supplemented with an initial condition, $u(x,
  y, 0)=u_0(x, y)$, $u_0:\Omega\rightarrow\mathbb{R}^m$,
  and different boundary conditions that may vary depending on the problem.

 \subsection{Finite difference WENO schemes}
 \label{ss:fdweno}
 Although the techniques that will be expounded in this paper are applicable
  to other numerical schemes,
  we use here Shu-Osher's finite difference conservative methods
  \cite{ShuOsher1989} with a
  WENO5 (\textit{Weighted Essentially
    Non-Oscillatory}) \cite{JiangShu96} spatial reconstruction,
  Donat-Marquina's
  flux-splitting \cite{DonatMarquina96}  and the RK3-TVD ODE
  solver \cite{ShuOsher89} in a method of lines fashion that we
  briefly describe here for the sake of  completeness. This
  combination of techniques was proposed in \cite{MarquinaMulet03}.

  We define our mesh starting from a reference vertical line,
  $x=\overline{x}$ and a horizontal one $y=\overline{y}$. Let $h_x>0$
  and  $h_y>0$ be the horizontal and vertical spacings of the mesh, so
  that the vertical lines in the mesh are determined by:
$x=x_r:=\overline{x}+rh_x$, $r\in\mathbb{Z}$ and
the horizontal ones by $ y=y_s:=\overline{y}+s h_y$,
$s\in\mathbb{Z}$. The cell with center $(x_r, y_s)$ is defined by:
\begin{equation*}
  [x_{r}-\frac{h_x}{2}, x_{r}+\frac{h_x}{2}]\times
  [y_{s}-\frac{h_y}{2}, y_{2}+\frac{h_y}{2}].
\end{equation*}
The computational domain is then given by
$$\mathcal{D}:=\left\{(x_r, y_s):\hspace{0.2cm}
(x_r,y_s)\in\Omega,\hspace{0.3cm}r,s\in\mathbb{Z}\right\}=
(\overline{x}+h_x\mathbb{Z})\times
(\overline{y}+h_y\mathbb{Z})\cap\Omega.$$
Notice that $\mathcal{D}$ is finite since  $\Omega$ is bounded.

  Shu and Osher's technique   \cite{ShuOsher89} to obtain high order
  finite difference schemes relies on the fact that, for fixed $y, t$:
  $$f(u)_x(x, y, t)=\frac{1}{h_x}\Big(\varphi(x+\frac{h_x}{2}, y ,t)-\varphi(x+\frac{h_x}{2}, y, t)\Big),
  $$
  for and implicitly defined $\varphi=\varphi^{h_x}$ that satisfies
  $$f(u(x,y,t))=\frac{1}{h_x}\int_{x-\frac{h_x}{2}}^{x+\frac{h_x}{2}}\varphi(\xi, y, t)d\xi.$$

  We can compute highly accurate approximations to
  $\varphi(x_{i\pm\frac{1}{2}}, y_j, t)$, denoted by
  $\widehat{\varphi}(x_{i\pm\mig}, y_j, t)$,
  using known grid values of $f(u)$ as
  \begin{align*}
  \varphi(x_{i+\frac{1}{2}}, y_j,
  t)&=\widehat{\varphi}(x_{i+\frac{1}{2}}, y_{j}, t)+\mathcal{O}
  (h_x^r), \\
  \widehat{\varphi}(x, y_{j}, t)&=\mathcal{R}(f_{i-p, j}(t),\ldots,f_{i+q, j}(t);x),\quad
  f_{l, j}(t)=f(u(x_l,y_j, t)),
\end{align*}
where $\mathcal{R}$ is a reconstruction procedure (a function whose
cell-averages coincide with the given ones; we use the WENO
  technique proposed in \cite{JiangShu96,SINUM2011} for five points
  stencils).
We can thus discretize the spatial derivative in \eqref{eq:hypsis} as:
\begin{equation}\label{eq:33}
\begin{aligned}
  (f(u))_x(x_{i}, y_{j}, t)=\frac{\hat{f}_{i+\frac{1}{2}, j}(t)-
    \hat{f}_{i-\frac{1}{2}, j}(t)}{h_x}+\mathcal{O}(h_x^r),\\
  \hat{f}_{i+\frac{1}{2}, j}(t)=\widehat{\varphi}(x_{i+\frac{1}{2}},
  y_j, t)=\hat{f}(u(x_{i-p}, y_{j}, t),\ldots,u(x_{i+q}, y_{j}, t)).
\end{aligned}
\end{equation}
Similarly
\begin{equation}\label{eq:34}
(g(u))_y(x_{i}, y_{j}, t)=\frac{\hat{g}_{i, j+\frac{1}{2}}(t)-
    \hat{g}_{i, j-\frac{1}{2}}(t)}{h_y}+\mathcal{O}(h_y^r).
\end{equation}

The numerical approximations in \eqref{eq:33} and \eqref{eq:34} must
be modified when some of the nodes involved in their computations
are not in $\Omega$. As we will detail in section \ref{sml} and \ref{sep},
we use extrapolation from interior data and possible boundary
conditions to obtain approximations
\begin{equation}\label{eq:35}
\widetilde{u}(x_r, y_s, t)=u(x_r, y_s, t)+\bigO(h_x^s),
\end{equation}
for  $(x_r, y_s)\notin\Omega$, should $u$ be defined and smooth enough
at an open set containing the stencil involved in the extrapolation
and the evaluation point $(x_r, y_s)$.
For the sake of the exposition, let us
assume that in \eqref{eq:33} $(x_{i-p-1}, y_j)\notin \Omega$ and
$(x_{i-p}, y_j),\dots, (x_{i+q}, y_j)\in \Omega$ and denote by
$v_{r}=u(x_r, y_{j}, t)$, $\widetilde{v}_{r}=\widetilde{u}(x_r, y_{j},
t)$.
  Then the
approximation in \eqref{eq:33} that corresponds to the numerical
scheme with the extrapolated value $\widetilde{u}(x_{i-p-1}, y_j, t)$
satisfies,  under mild smoothness conditions,
\begin{multline}\label{eq:36}
\frac{\hat{f}(v_{i-p},\ldots,v_{i+q})-
  \hat{f}(\widetilde{v}_{i-p-1},\ldots,v_{i+q-1})
}{h_x}\\
=
\frac{\hat{f}(v_{i-p},\ldots,v_{i+q})-
  \hat{f}(v_{i-p-1},\ldots,v_{i+q-1})
}{h_x}+\bigO(h_x^{s-1})\\=
(f(u))_x(x_{i}, y_{j}, t)+\mathcal{O}(h_x^{r'}),
\end{multline}
with $r'=\min(r, s-1)$, as it  will be seen in Appendix \ref{ap:accuracy}. The
same conclusion holds if other arguments are replaced by suitable
extrapolations.
We will see in the numerical experiments how the choice of $s$ affects
the accuracy of the solution. The issue of
stability of the  scheme with extrapolation at the boundary will be
tackled in Appendix  \ref{ap:stability} in a simple  one-dimensional setup
that will serve as guide for the design of the multidimensional
extension of extrapolation.

  The spatial discretization of  problem \eqref{eq:hypsis} reads as:
  \begin{align*}
  u_{i,j}'(t)&=-\mathcal{L}(u(t))_{i, j},\quad
  u_{i,j}=[u_{1,i,j},\ldots,u_{m,i,j}]^T,\\
  \mathcal{L}(u)_{i, j}&=g
  \frac{\hat{f}_{i+\frac{1}{2},j}-\hat{f}_{i-\frac{1}{2},j}}{h_x}+
  \frac{\hat{g}_{i,j+\frac{1}{2}}-\hat{g}_{i,j-\frac{1}{2}}}{h_y},
\end{align*}
for approximations $u_{i, j}(t)\approx u(x_i,y_j, t)$ and
it can be solved using an   appropriate ODE solver. In this paper, we
use the third order TVD
  Runge-Kutta scheme proposed in \cite{ShuOsher89}:
  \begin{equation}\label{rk3tvd}
    \left\{\begin{array}{rcl}
      u^{(1)}&=&u^n-\Delta t\mathcal{L}(u^n),\\
      u^{(2)}&=&\frac{3}{4}u^n+\frac{1}{4}u^{(1)}-\frac{1}{4}\Delta
      t\mathcal{L}(u^{(1)}),\\
      u^{n+1}&=&\frac{1}{3}u^n+\frac{2}{3}u^{(2)}-\frac{2}{3}\Delta
      t\mathcal{L}(u^{(2)}),\\
    \end{array}\right.
  \end{equation}
  for $u_{i, j}^n\approx u_{i,j}(t_n)\approx u(x_i, y_j, t_n)$.

  To extend these schemes to systems of conservation laws
  we can compute the numerical flux $\hat{f}_{i+\frac{1}{2}}$ (we drop
  the $j$ subscript for simplicity) by using
  a fifth order Donat-Marquina's flux-splitting \cite{DonatMarquina96}:
  \begin{equation}\label{char}
    \begin{split}
    \hat{f}_{i+\frac{1}{2}}=\sum_{k=1}^mr^{+,k}\left(\mathcal{R}^+
    \left(l^{+,k}\cdot f_{i-2}^{+,k},\ldots,l^{+,k}\cdot
      f_{i+2}^{+,k};x_{i+\frac{1}{2}}\right)\right) \\
  +\sum_{k=1}^mr^{-,k}\left(\mathcal{R}^-
    \left(l^{-,k}\cdot f_{i-1}^{-,k},\ldots,l^{-,k}\cdot
      f_{i+3}^{-,k};x_{i+\frac{1}{2}}\right)\right),
    \end{split}
  \end{equation}
  where $f^{\pm,k}_l=f^{\pm,k}(u_l)$ as defined below,
  $r^{\pm,k}=r^k(u_{i+\frac{1}{2}}^{\pm}),l^{\pm,k}=l^k(u_{i+\frac{1}{2}}^{\pm})$
  are the right and left normalized eigenvectors corresponding to the
  eigenvalue $\lambda_k(f'(u_{i+\frac{1}{2}}^{\pm}))$ of the flux
  Jacobian $f'(u_{i+\frac{1}{2}}^{\pm})$, respectively, computed at
  $u_{i+\frac{1}{2}}^{\pm}$, where
  $$u_{i+\frac{1}{2}}^{+}=\mathcal{I^+}(u_{i-2},\ldots,u_{i+2};x_{i+\frac{1}{2}}),
  \hspace{0.3cm}
    u_{i+\frac{1}{2}}^{-}=\mathcal{I^-}(u_{i-1},\ldots,u_{i+3};x_{i+\frac{1}{2}}),$$
    for some interpolators $I^{\pm}$.
  The functions $f^{\pm,k}$  satisfy $f^{+,k}+f^{-,k}=f$,
  $\pm\lambda_k((f^{\pm,k})'(u))>0$ for $u$ in some relevant range
  $\mathcal{M}_{i+\frac{1}{2}}$ near $u_{i+\mig}^{\pm}$, and are given by:
  \begin{align*}
    (f^{-,k}, f^{+,k})(v)&=
    \begin{cases}
      (0, f(v)) & \lambda_k(f'(u)) > 0,\quad\forall u\in \mathcal{M}_{i+\mig}\\
      (f(v), 0) & \lambda_k(f'(u)) < 0,\quad\forall u\in \mathcal{M}_{i+\mig}\\
      (\mig(f(v)- \alpha_{i+\frac{1}{2}}^{k} v), \mig(f(v)+ \alpha_{i+\frac{1}{2}}^{k} v)),&\exists u\in
      \mathcal{M}_{i+\mig} /  \lambda_k(f'(u)) =      0,
    \end{cases}
  \end{align*}
 where $\alpha_{i+\frac{1}{2}}^{k}\geq |\lambda_k(f'(u))|$ for $u\in \mathcal{M}_{i+\frac 12}$. For the
  Euler equations we can simply take $\mathcal{M}_{i+\frac 12}=\{u_i, u_{i+1}\}$.

\section{Meshing procedure}\label{sml}
\subsection{Ghost cells}
We recall that WENO schemes of order $2k-1$ use an stencil
(consecutive indexes) of  $2k$
points, therefore $k$ additional cells are needed at both sides of
each horizontal and vertical mesh line in order to perform a time
step. These additional cells are usually named \textit{ghost cells} and, in
terms of their centers, are given by:
$$\mathcal{GC}:=\mathcal{GC}_x\cup\mathcal{GC}_y,$$
where
$$\mathcal{GC}_x:=\left\{(x_r, y_s):\hspace{0.2cm}0<d\left(x_r,\hspace{0.1cm}\Pi_x\left(\mathcal{D}\cap\left(\mathbb{R}\times\{y_s\}\right)\right)\right)\leq kh_x,\hspace{0.3cm}r,s\in\mathbb{Z}\right\},$$
$$\mathcal{GC}_y:=\left\{(x_r,
  y_s):\hspace{0.2cm}0<d\left(y_s,\hspace{0.1cm}\Pi_y\left(\mathcal{D}\cap\left(\{x_r\}\times\mathbb{R}\right)\right)\right)\leq
  kh_y,\hspace{0.3cm}r,s\in\mathbb{Z}\right\},$$
where $\Pi_x$ and $\Pi_y$ denote the projections on the respective
coordinates and,
$$d(a, B):=\inf\{|b-a|:\hspace{0.1cm}b\in B\},$$
for given   $a\in\mathbb{R}$ and
$B\subseteq\mathbb{R}$.
Notice that
$d(a,\emptyset)=+\infty$, since, by convention, $\inf\emptyset=+\infty$.

\subsection{Normal lines}\label{ss:nl}
The extrapolation procedure in 1D can be reasonably performed after
the accuracy and stability analysis that have been performed in
Section \ref{seq} and  Appendix \ref{ap:stability}. We focus now on the
two-dimensional setting and boundaries with prescribed Dirichlet
conditions, e.g., reflective boundary conditions for the Euler
equations. In this situation, it seems reasonable that the
extrapolation at a certain ghost cell $c_*=(x_*, y_*)$ be based on the
prescribed value at the nearest boundary point.
It can be proven that a point  $p\in\partial\Omega$ satisfying
$$\|c_*-p\|_2=\min\{\|c_*-p'\|_2:\quad p'\in\partial\Omega\}$$
also satisfies that the line determined by  $c_*$ and $p$ is
normal to the curve $\partial\Omega$ at  $p$, if $\partial\Omega$ is
differentiable at $p$. Uniqueness of $p$
holds whenever $c_*$ is close enough to the boundary, so
 we will
henceforth  denote $N(c_*)=p$.

This argument suggests that a good strategy is to perform a (virtual)
rotation of the domain and obtain data on some
points $N_i\in\Omega$ on  the line that
passes through $c_*$ and $N(c_*)$ (normal line to $\partial \Omega$
passing by $c_*$)
(see Subsection \ref{sss:choice_nodes_nl} for the details on this choice)
and then use a one-dimensional extrapolation from the data on these
points on the segment   to approximate the value at $c_*$.

In case that the boundary conditions prescribe values for the normal
component of a vectorial unknown $\overrightarrow{v}$ related to the
coordinates frame (as is the case
for reflective
boundary conditions for the Euler equations), then one defines
$$\overrightarrow{n}=\frac{c_*-N(c_*)}{\Vert c_*-N(c_*)\Vert},\quad
\overrightarrow{t}=\overrightarrow{n}^{\perp},
$$
and obtains normal
and tangential components of $\overrightarrow{v}$ at each point $N_i$ of the mentioned
segment by:
\begin{equation*}
  v^t(N_i)=\overrightarrow{v}(N_i)\cdot \overrightarrow{t}, \quad
  v^n(N_i)=\overrightarrow{v}(N_i)\cdot \overrightarrow{n}.
\end{equation*}
The extrapolation procedure is applied to $v^t(N_i)$ to approximate
$v^t(c_*)$  and to $v^n(N_i)$  and $v^n(N(c_*))=0$ to approximate
$v^n(c_*)$. Once $v^t(c_*), v^{n}(c_*)$ are approximated, the
approximation to $\overrightarrow{v}(c_*)$ is set to
\begin{equation*}
  \overrightarrow{v}(c_*)=v^{t}(c_*)\overrightarrow{t}+v^{n}(c_*)\overrightarrow{n}.
\end{equation*}

\subsubsection{Choice of nodes on normal lines}
\label{sss:choice_nodes_nl}
 As mentioned in Section \ref{seq},
if we wish to formally preserve a certain precision in the resulting
scheme, it is
necessary to extrapolate the information from the interior of the
domain in an adequate manner. Therefore, if the basic numerical scheme
has order $r$ it is reasonable to use extrapolation of this order at
least. For the sake of clarity, we will not distinguish between
interpolation or extrapolation when these take place at the interior
of the domain.

At this point there are many possibilities. However, as expected, not
all of them yield the same quality in the results nor the same
computational efficiency. The following configuration supposes a
reasonable balance between both previous factors.

We proceed in a similar fashion as in
  \cite{Sjogreen}. Let $(x_*,y_*)\in\mathcal{GC}$ and consider the
corresponding point in
$\partial\Omega$ at minimal distance, $N(x_*,y_*)$. As already mentioned,
the vector determined by both points is orthogonal to
$\partial\Omega$ at
$N(x_*,y_*)$. Let us suppose that we wish to use an extrapolation of order
$r$ at the ghost cell  center
$(x_*,y_*)$.

At first place, one needs to obtain data from the information in
$\mathcal{D}$ at a set of points $\mathcal{N}(x_*,
y_*)=\{N_1,\dots,N_R\}$, with $R\geq r$, on the line
determined by the points  $(x_*,y_*)$ and $N(x_*,y_*)$. By a CFL
stability motivation (see  Appendix \ref{ap:stability} for further details), we
will do the
selection with a spacing  between them of at least the distance
between $(x_*, y_*)$ and $N(x_*, y_*)$. We will go further
  into this issue in section \ref{srn}.
 We will choose the nodes  depending on the slope of the normal line, so
 that the use of interior information is maximized. We will henceforth
 denote by $E(x)$ the integer rounding of $x$ towards $\infty$, i.e.,
 $\text{sign}(E(x))=\text{sign}(x)$ and $|E(x)|=\min (\mathbb N \cap
 [|x|, \infty))$.

We denote by $v=(v_1,v_2)$ the vector determined by $(x_*,y_*)$ and
$N(x_*,y_*)$, so that the normal line passing through $(x_*, y_*)$ is
given by the parametric equations:
\begin{align*}
  x&=x_*+s v_1\\
  y&=y_*+s v_2.
\end{align*}
Depending on the angle  $\theta$ of the vector
$v=(v_1,v_2)$, we consider two possibilities:
\begin{enumerate}
\item $|v_1|\geq |v_2|$.
\item $|v_1| < |v_2|$.
\end{enumerate}
In the first case we take points $N'_q=(x_*+qC_xh_x,
y_*+qC_xh_x\frac{v_2}{v_1})$, with $C_x\in\mathbb Z$ chosen with the
same sign as
$v_1$ and so that:
\begin{align*}
  \Vert N'_q-N'_{q+1}\Vert_2 &\geq  \Vert v\Vert_2\\
  \Vert N'_q-N'_{q+1}\Vert_2&=\frac{h_x|C_x|}{|v_1|}\Vert v\Vert_2 \geq
\Vert v\Vert_2 \Leftrightarrow |C_x|\geq \frac{|v_1|}{h_x}.
\end{align*}
Therefore, our choice is $C_x=E(\frac{v_1}{h_x})$.

  Now, if Dirichlet
boundary conditions at $N(x_*, y_*)$ are prescribed, we take
  the nodes
  \begin{equation}\label{eq:nodesinflow}
    \mathcal{N}(x_*,y_*)=\{N(x_*, y_*), N'_2, \dots, N'_{R+1}\},
  \end{equation}
  i.e., we drop $N'_1$ from the
  list. If no boundary 
  condition is specified on $N(x_*, y_*)$ then
  \begin{equation}\label{eq:nodesoutflow}
    \mathcal{N}(x_*,y_*)=\{N'_1, N'_2, \dots, N'_{R}\}.
  \end{equation}
In this fashion, the chosen nodes
$\mathcal{N}(x_*,y_*)=\{N_1,\dots,N_{R}\}$ satisfy   $\Vert N_q-N_{q+1}\Vert_2
\geq  \Vert v\Vert_2$, $q=1,\dots,R-1$.

Denote $N_q=(\widetilde x_q,\widetilde y_q)$. For each $q$ for which
$u(\widetilde x_q,\widetilde y_q)$ is not known, we need to obtain a
sufficiently accurate  approximation of this value from the information on the
interior nodes.
Since the second coordinate, $\widetilde y_q$, of $N_q$  does not need to coincide with
the center of a vertical     cell, we will use interpolation  from the
cells in  the line  $x=\widetilde x_q$ by using the following set of points:
\begin{align*}
  \mathcal{S}_q&=\{N_{q,1},\dots,N_{q,R}\}:=\argmin_{A\in\mathcal{A}}
    \sum_{(\widetilde x_q, y_s) \in  A}|y_s-\widetilde y_q|,
    \\
    \mathcal{A}&:=\{ A=\{(\widetilde x_q, y_j),\dots,(\widetilde x_q,
    y_{j+R-1})\} / A \subseteq \mathcal{D}
\}.
\end{align*}
That is, we select the vertical stencil of length $R$ with a  first
coordinate fixed to  $\widetilde x_q$ such that it be as centered as
possible with respect to the point $N_q$, see
    Figure \ref{fig:Cxy} (a) for a graphical example.

    In a dual fashion, in the second case we take points
$N_q=(x_*+qC_yh_y\frac{v_1}{v_2}, y_*+qC_yh_y)$, with
$C_y=E(\frac{v_2}{h_y})$
and
\begin{align*}
  \mathcal{S}_q&=\{N_{q,1},\dots,N_{q,R}\}:=\argmin_{A\in\mathcal{A}}
    \sum_{( x_s, \widetilde y_q) \in  A}|x_s-\widetilde x_q|,
    \\
    \mathcal{A}&:=\{ A=\{(x_j, \widetilde y_q),\dots,(x_{j+R-1}, \widetilde y_q)\} / A \subseteq \mathcal{D}
\}.
\end{align*}
See     Figure \ref{fig:Cxy} (b) for a graphical example.

    \begin{figure}[htb]
      \centering
      \begin{tabular}{cc}
      \includegraphics[width=0.4\textwidth]{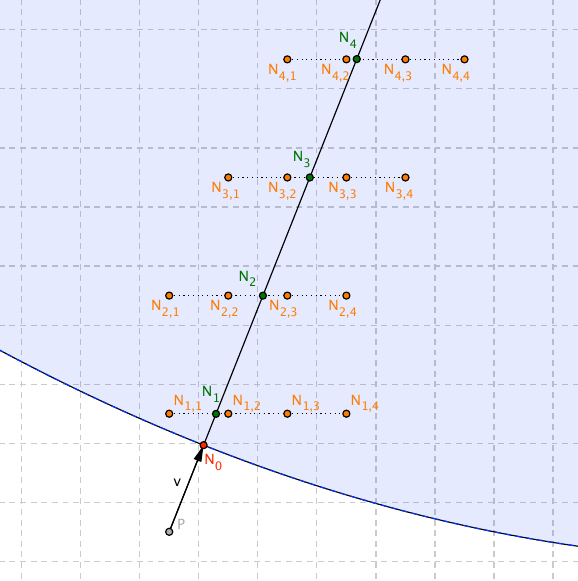} &
      \includegraphics[width=0.4\textwidth]{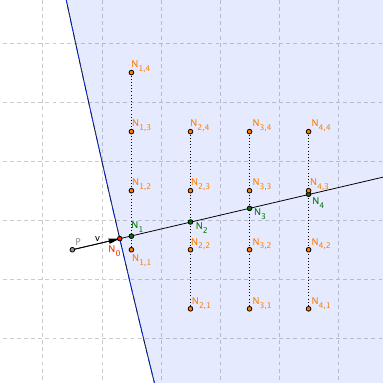}\\
      (a) & (b)
    \end{tabular}
    \caption{Examples of choice of stencil: (a) $C_y=2$, $N_{q,i}\in\mathcal{S}_q$; (b) $C_x=1$, $N_{q,i}\in\mathcal{S}_q$.}
      \label{fig:Cxy}
    \end{figure}

The above procedure for the selection of the interpolation nodes at
the normal lines and their corresponding sets $\mathcal{S}_q$ is
performed only once at the beginning of the simulation as long as the
boundary does not change. With an adequate use of this data structure,
one can reconstruct data
at order $r$ (in case of smoothness) at the points $N_1,\dots,N_R$ on
the normal line. Once these values are obtained, they are used  to
finally extrapolate  to the given ghost cell $(x_*, y_*)$.
The full
extrapolation procedure is thus done in two stages in general: in the first
 one data located at the normal lines is computed from
 the numerical solution by (horizontal or vertical) 1D
 interpolation; in the second one values for the ghost cells
 are obtained by 1D extrapolation along the normal line
 from the data in the normal line obtained in the first stage
 and eventually data coming from the boundary conditions.

To perform these  two one-dimensional data approximations,
it should be taken
into account that the selected stencils can include regions with
singularities. We will see in the next section how to proceed in this
case.

For our one-dimensional experiments, the corresponding nodes selection
procedure  is  as follows:
 Given a ghost cell $x_*$, we denote $v=N(x_*)-x_*$ and  take
points $N_q=x_*+qC_xh_x$, with $C_x=E(\frac{v}{h_x})$, so that
\begin{align*}
  | N_q-N_{q+1}| = |C_x|h_x &\geq  | v|.
\end{align*}

  Now, if Dirichlet
boundary conditions at $N(x_*)$ are prescribed,
we take   the nodes
  \begin{equation}\label{eq:nodesinflow1d}
    \mathcal{N}(x_*)=\{N(x_*), N'_2, \dots, N'_{R+1}\},
  \end{equation}
  i.e., we drop $N'_1$ from the
  list. If no boundary 
  condition is specified on $N(x_*)$ then
  \begin{equation}\label{eq:nodesoutflow1d}
    \mathcal{N}(x_*)=\{N'_1, N'_2, \dots, N'_{R}\}.
  \end{equation}
As in the general case, extrapolation based on $N_1,\dots,N_R$ is performed to
assign data to $x_*$.

\section{Extrapolation}\label{sep}
\subsection{Motivation}\label{sec:motivation}
Interpolation can produce large errors if there is a discontinuity in
the region determined by the interpolation nodes and the evaluation
point. When implementing extrapolation at ghost cells, as in Section
\ref{sml},   in order to avoid this considerable loss of precision or even a
complete failure of the simulation, it is necessary to handle this
situation carefully.

Since in our procedure the interpolator is evaluated at a point which is
not necessarily centered with respect to the interpolation nodes, we
cannot directly use  techniques
based on the partition of the stencil in substencils and/or the
weighting of these, such as it is done in ENO
  \cite{Harten1987} or WENO schemes, since
in this case
not all the substencils are useful, this depending on the localization
of the discontinuity and the evaluation point.

Consider for instance the function $u:=\chi_{[1/2,+\infty)}$ and
take the nodes $x_i:=i-1,$ $1\leq i\leq 5$, with nodal values $u_i:=u(x_i)$ and
suppose we want to extrapolate this information at $x^*=-1$. Our nodal
values are thus $u_1=0$ and $u_i=1$ for $2\leq i\leq 5$. It is well-known that the
ENO3 technique divides
the global stencil of five points into three substencils,
$S_r:=\{x_{r+1},x_{r+2},x_{r+3}\},$ $0\leq r\leq 2$, and chooses the one
with maximal smoothness in terms of its divided differences, in this case,
$S_1$ or $S_2$, both with all nodal values equal to 1 and thus all
derivatives are zero. However, the result of this extrapolation at
$x^*=1$ is 1, which corresponds to the other state of the discontinuity
from where $x^*$ is located. The same applies for WENO.

Therefore, the interpolation strategy should be made more
flexible, in order to choose certain nodes as valid
according to some criterion and reject the rest. The strategy expounded
in section  \ref{sml} lets us focus on a one-dimensional setting.

\subsection{Stencil selection by thresholding}
Lets us assume that we have information on a stencil of not necessarily
equispaced nodes, $x_1 < \dots < x_R$, with corresponding nodal values
$u_i=u(x_{i})$, and that we wish to interpolate at a
certain node $x^*$.

The key node on which we establish a proximity criterion on its
corresponding nodal value is the interior node which is the closest
to $x^*$, i.e.,
we choose the node $x_{i_0}$, $i_0\in\{1,\ldots,R\}$ such that:
$$i_0=\argmin_{1\leq i\leq R}|x_i-x^*|.$$
Now,  the goal is to  approximate the value that that node should
have, based on the information of the ``smoothest'' substencil and the
node $x_{i_0}$.

Let $M$, $1\leq M\leq R$, be the prescribed size for  substencils. We
therefore have  $R-M+1$ possible substencils:
$$S_r=\{x_{r+1},\dots,x_{r+M}\},\quad 0\leq r\leq R-M.$$
We denote by $p_r(x)$ the interpolator associated to the stencil
$S_r$, $0\leq r\leq R-M$. If sufficient smoothness at the whole
stencil holds, then one has:
\begin{equation}\label{ordr}
  \left.\begin{array}{rcl}
  u(x_i)-p_r(x_i)&=&\mathcal{O}(h_x^M),\quad i=1,\dots,R,
  \end{array}\right.,
\end{equation}
therefore
\begin{equation*}
  u(x_i)=u(x_{i_0})+(p_r(x_i)-p_r(x_{i_0}))+\mathcal{O}(h_x^M).
\end{equation*}
We select the  substencil that solves:
\begin{equation}\label{eq:choice_r0}
r_0:=\argmin_{0\leq r\leq
  R-M}\sum_{k=1}^{M-1}
\int_{x_{r+1}}^{x_{r+M}}(x_{r+M}-x_{r+1})^{2k-1}p_r^{(k)}(x)^2dx,
\end{equation}
and define
\begin{equation}\label{vi}
  v_i:=u_{i_0}+(p_{r_0}(x_i)-p_{r_0}(x_{i_0})).
\end{equation}
 From  (\ref{ordr}), we have that $v_i=u_i+\mathcal{O}(h_x^M)$
 if there's smoothness up to the $(M-1)$-th derivative.
On the other hand, assuming that $u$ is smooth on an open set that contains
$S_{r_0}$, if there is a discontinuity within
the whole stencil $\bigcup_{r=0}^{R-M}S_r$ and
$u_i$ is quite far from $u_{i_0}$, since by construction
$v_i=u_{i_0}+\mathcal{O}(h_x)$, then it can be expected that
 $v_i$ also be quite different from $u_i$.

In order for the smoothness assumption on $S_{r_0}$ to make sense in a
general setting, one needs $M\leq
E\left(\frac{R}{2}\right)$, because all substencils would
overlap in some common central nodes otherwise, leading to a situation where all substencils contain a discontinuity if
it is contained in the overlapping region.

Therefore, one can conclude that the proximity   of  $v_i$ with
respect to  $u_i$ indicates the stencil smoothness that would entail
including $u_i$ in a hypothetical stencil for the final extrapolation.

Finally, let $\delta \in (0, 1]$ be a \textit{threshold} and define
the set of indexes
\begin{equation}\label{eq:Idelta}
  I_{\delta}:=\left\{i\in\{1,\ldots,R\}:\quad
\delta\left(|u_i-u_{i_0}|+D(x_i)\right)
\leq|v_i-u_{i_0}|+D(x_i)\right\},
\end{equation}
where
$$D(x):=\sum_{j=1}^{M-1}\left|(x-x_{i_0})^jp_{r_0}^{(j)}(x_{i_0})\right|.$$
Notice that  $I_{\delta}\neq\emptyset$, since $i_0\in
I_{\delta}$.

 The term $D(x_i)$ is used in \eqref{eq:Idelta} to avoid an order loss at smoothness regions whenever
$\exists j_0$, $1\leq j_0\leq
M-1:$ $|u^{(j_0)}|\geq\mathcal{O}(\Delta x^M)$ near
  $x_{i_0}$. When the first derivative is
close to zero, despite both $|v_i-u_{i_0}|$ and $|u_i-u_{i_0}|$ are still
$M$-th order close, its quotient may be far from 1, specially when one of
the previous expressions is close to be zero or even exactly zero (for
instance, in zeros of even degree functions). The terms
$D(x_i)$, alleviates this discrepancy by
adding a
$\mathcal{O}(h_x^{k_0})\neq\mathcal{O}(h_x^{k_0+1})$ term, with
$k_0<M$, with $k_0$ the minimum index such that the $k_0$-th
derivative does not have a zero around $x_{i_0}$. The definitive
stencil is the largest stencil in $I_\delta$ containing $i_0$.

As last (optional) filter, if  $u^*$ is the value obtained from
Lagrange interpolation from the resulting stencil,
then the same threshold criterion can be applied to that value,
resulting in the definitive extrapolation value:
\begin{equation}\label{eq:udef}
u^*_{\textnormal{def}}=\left\{\begin{array}{ccc}
    u^* & \textnormal{if} &
    \delta'\left(|u^*-u_{i_0}|+D(x^*)\right)
    \leq|p_r(x^*)-p_r(x_{i_0})|+D(x^*)
    \\
    u_{i_0} & \textnormal{if} &
    \delta'\left(|u^*-u_{i_0}|+D(x^*)\right)
    >|p_r(x^*)-p_r(x_{i_0})|+D(x^*)
    \\
    \end{array}\right.\end{equation}
with $0\leq\delta'\leq1$.

This last criterion can be useful to detect wrong extrapolations (even
when data are apparently smooth and previous criteria are met).
Since it is an a posteriori criterion, we may generally use it with
threshold values that are more permissive, i.e., much smaller than one, than those for the nodes
rejection.
By construction, the closer the parameter   $\delta$ is to one the
lesser the tolerance to high gradients will be (with the consequent
risk of eliminating some nodes from smooth regions). On the other
hand, if $\delta$ is set too low, there may appear some oscillations
or artifacts near discontinuities.

The quality of the smoothness criterion is enhanced with a larger
substencil size (there is less risk of rejecting ``correct''
nodes). Furthermore a larger substencil can be also used to avoid a
loss of precision order when consecutive derivatives are null at some
point (precisely, until the $(M-2)$-th order derivative).
Nevertheless, this would force increasing the size $R$ of the
global stencil, i.e., obtain more data from the general problem in
order to avoid the previously mentioned problem.

In summary, the extrapolation of the nodal data
$\{(x_i,u_i)\}_{i=1}^R$ to the point $x^*$ consists of the following steps:
\begin{enumerate}
  \item Find $i_0$ such that $x_{i_0}$ is the closest node to $x^*$.

  \item Find the $M$-point stencil $\mathcal{S}_{r_0} = \{x_{r_0+j}\}_{j=1}^{M}$ with
    maximal smoothness. We use the smoothness indicators in \eqref{eq:choice_r0}
     for this purpose.

  \item For $i\in \{1,\dots, R\}$ compute candidate
  approximations $v_i$ of $u_i$ using \eqref{vi}.
  \item Fix a value $0<\delta\leq1$ and compute the set of nodes
$I_\delta$ according to \eqref{eq:Idelta}.

\item Extract the largest stencil in $I_\delta$ containing $i_0$.

\item Compute the extrapolated value $u^*$ at $x^*$ using the stencil in the previous step.

 \item Optionally, fix $0<\delta'\leq1$ and replace $u^*$ by $u^*_{\textnormal{def}}$ 
  computed from \eqref{eq:udef}.
\end{enumerate}

Let us apply the previous steps to the toy example in Section \ref{sec:motivation}.
We have $R=5, M=3$ in that example and we assume the values $\delta=\delta'=0.5$, although any
other choice of $\delta$ and $\delta'$ in the range $(0,1)$ would give the same result.
The stencil selection procedure is as follows:
\begin{enumerate}
\item The closest node to $x^*=-1$ is $x_1=0$, whose nodal value is
$u_1=0$.
\item There are two stencils where the information is constant, $S_1$
and $S_2$ and therefore any of them would be selected in this step
leading to the same result. Assume $S_1$ is chosen.

\item $v_i = u_1 = 0$ $1\leq i\leq 5$ because $p_1 = 1$ for all $i\in \{1, \dots, 5\}$ and thus $D(x)=0$.

\item The differences $|u_i-v_i|$ are all equal to 1 except for $x_1$ for which it is equal to 0.
Therefore $I_\delta = \{x_1\}$ and the result of the extrapolation is $u^*=u_1=0$.
\item If the a posteriori filter is applied the result is kept as $\delta'|u^*-u_1|=0\leq 0=|p_{1}(x^*) - p_1(x_1)|$.
 \end{enumerate}

 In the numerical experiments that we will include in section
\ref{srn}, specially in
problems with shocks, contacts and turbulence, there are occasions in
which we will have to select more nodes than the minimum that would be
needed to attain the order of the method.  The reason for this
decision is  to yield more search room to locate a substencil with the
maximal smoothness, for discontinuities propagate through several
cells and one should compensate for this additional margin.

The well-posedness of the initial-boundary value problems depends
heavily on the correct setting of boundary conditions. For linear
systems and some nonlinear systems, it is well-known that boundary
conditions on incoming characteristic variables (either globally or
locally defined) should be obtained from prescribed boundary data (on
primitive or conserved variables) and outgoing characteristic
variables,  in such a way that the number of incoming characteristics
is the number of boundary conditions
(see
\cite{GodlewskiRaviart96,GustafssonKreissOliger95,OligerSundstrom78}
and references therein).
This is the approach for setting
numerical boundary conditions in general systems that has been used in  \cite{Sjogreen,TanShu,TanWangShu}. In practice
 a formulation based on either conserved or primitive variables is preferred
 as long as the possible combinations of (conserved or primitive) variables to be determined by boundary conditions are known.
 For the Euler equations this equivalence is known and depends essentially on whether the flow
 is subsonic or supersonic and the kind of boundary conditions under consideration. We refer to
 \cite{Yee1981,PletcherEtAl2012} for further details.

 In the examples shown in section \ref{srn}
 we focus on the case of the Euler equations with reflecting boundary
 conditions on  solid walls, at which the normal velocity is
 set to zero, thus only one characteristic is
 incoming. Furthermore, the value corresponding to this incoming
 characteristic can be obtained from the normal velocity (which is set
 to zero) and  the other
 outgoing characteristics. This is consistent with our procedure for
 the 2D experiments, for which we use
 extrapolation on rotated primitive variables $\rho, v^{n}, v^{t},p$
 (see Section \ref{ss:nl} for the definition of $v^{n}, v^{t}$), using
 boundary data $v^{n}=0$ on the solid walls.

 On the other hand, when the flow is
 supersonic all characteristics are incoming at the inlet (inflow boundary)
 and outgoing at the outlet (outflow boundary)
and therefore we fix all primitive variables at
 the inflow boundaries and none at the outlet.
 Therefore, depending on the type of boundary condition, some variables may be fixed to values determined
 by the boundary conditions at the boundary points, which are included
 in the stencil used for extrapolation. A proper node spacing for
 stability in ensured as described in Appendix \ref{ap:stability}.

\section{Numerical experiments}\label{srn}

\subsection{One-dimensional experiments}
\label{ss:1dexperiments}

In this section we present some one-dimensional
  numerical experiments where both the accuracy of the extrapolation method 
  for smooth solutions and its behavior in  presence of discontinuities
   will be tested and analyzed.
      
Let us remark that for one-dimensional tests it is not necessary to
develop a procedure as in the two-dimensional case described in Section \ref{sml}, because one can
set up initially a proper spacing between the nodes and perform a
straight extrapolation at the ghost cells without having stability issues
due to the presence of small-cut cells. However,
to present accuracy and stability analysis in an easier setup,
we perform the one-dimensional extrapolation explained in
\eqref{eq:nodesinflow1d}, \eqref{eq:nodesoutflow1d}, which directly corresponds to the two-dimensional
extrapolation procedure that is proposed in this paper.  
This approach will illustrate that the accuracy order will still be
the expected one in  
the smooth case and that the extrapolation method shows good 
performance in the non-smooth case. 

\subsubsection{Linear advection,  $\mathcal{C}^{\infty}$ solution.}

We start with a simple one-dimensional test case
that will be used
to illustrate the performance of the proposed method
and also to analyze the importance and relative influence of some
elements of the algorithm along the four examples detailed below.
The problem statement for this test is the same as in \cite{TanShu}.
We consider the linear advection equation
$$u_t+u_x=0,\quad \Omega:=(-1,1),$$
with initial condition given by $u(x,0)=0.25+0.5\sin(\pi x)$ and
boundary condition
$u(-1, t)=0.25-0.5\sin(\pi(1+t)),$ $t\geq0$. We apply a numerical
outflow condition at $x=1$, where
Dirichlet boundary conditions cannot be imposed due to the direction
of propagation of the information.

It is immediately checked that the unique  (smooth) solution to this problem
is $$u(x,t)=0.25+0.5\sin(\pi(x-t)).$$

{\bf Example 1.} In order to numerically test the order of accuracy
we perform tests at resolutions  given by $n=20 \cdot 2^j$ points,
$j=1,\dots,5$. The cell centers are
$x_j:=-1+(j+\frac{1}{2})h_x$, with $h_x:=\frac{2}{n}$. We
recall that the set of all cell centers which are interior to $\Omega$
is
$$\mathcal{D}:=\left\{x_j:\hspace{0.1cm}j\in\{0,\ldots,n-1\}\right\}.$$
Since we use WENO5 reconstruction, we require 3 extra cells at each
side of the boundary, where extrapolation from the interior will take place.
\begin{itemize}
  \item $x=-1$: $x_j,$ $-3\leq j\leq -1$.
  \item $x=1$: $x_j,$ $n\leq j\leq n+2$.
\end{itemize}

We perform the one-dimensional extrapolation explained in
\eqref{eq:nodesinflow1d}, \eqref{eq:nodesoutflow1d}, which directly corresponds to the two-dimensional
extrapolation procedure that is proposed in this paper.

Given that the ODE solver is third order accurate, in order to attain
fifth order accuracy in the overall scheme, we need to select a time
step given by  $\Delta
t=\left(\frac{2}{n}\right)^{\frac{5}{3}}$, which corresponding Courant
numbers $\Delta t / h_x=(2/n)^{2/3}\leq 1/20^{2/3}$.

Since the left boundary conditions are time dependent,
we also have to take into account that a specific approximation is
needed in each of the 3 stages in each  RK3-TVD time step. In general,
if the inflow condition is given by some function $g(t)$ which is at
least twice continuously differentiable, we have to use the following
values at the boundary to preserve third order accuracy \cite{Carpenter}:
\begin{itemize}
  \item First stage: $g(t_k).$
  \item Second stage: $g(t_k)+\Delta tg'(t_k).$
  \item Third stage: $g(t_k)+\frac{1}{2}\Delta tg'(t_k)+\frac{1}{4}\Delta
    t^2g''(t_k).$
\end{itemize}
Taking into account all the previous considerations, we execute the
simulation until  $t=1$ for all the previously
specified resolutions and we study the errors in the 1 and $\infty$
norms, together with the order deduced from them. We consider
different modalities of boundary extrapolation: Constant extrapolation
using only the closest node value (Table \ref{O1}), five points
stencil Lagrange extrapolation
without discontinuity filters (Table \ref{lagrange}) and with
filters by thresholding for different choices of the
  thresholds (Tables \ref{detec99}--\ref{detec755}).
\begin{table}[htb]
  \centering
  \begin{tabular}{|c|c|c|c|c|}
    \hline
    $n$ & Error $\|\cdot\|_1$ & Order $\|\cdot\|_1$ & Error
    $\|\cdot\|_{\infty}$ & Order $\|\cdot\|_{\infty}$ \\
    \hline
    40 & 2.07E$-3$ & $-$ & 3.87E$-2$ & $-$  \\
    \hline
    80 & 5.32E$-4$ & 1.96 & 1.96E$-2$ & 0.98 \\
    \hline
    160 & 1.34E$-4$ & 1.99 & 9.81E$-3$ & 1.00 \\
    \hline
    320 & 3.38E$-5$ & 1.99 & 4.91E$-3$ & 1.00 \\
    \hline
    640 & 8.48E$-6$ & 1.99 & 2.45E$-3$ & 1.00 \\
    \hline
  \end{tabular}
  \caption{Constant extrapolation (first order).}
  \label{O1}
\end{table}

The Table \ref{O1} illustrates that a low order extrapolation affects the
order of the global scheme. We can see that in this case is downgraded
to second order in $\|\cdot\|_1$, while it's first order in
$\|\cdot\|_{\infty}$.
\begin{table}[htb]
  \centering
  \begin{tabular}{|c|c|c|c|c|}
    \hline
    $n$ & Error $\|\cdot\|_1$ & Order $\|\cdot\|_1$ & Error
    $\|\cdot\|_{\infty}$ & Order $\|\cdot\|_{\infty}$ \\
    \hline
    40 & 8.73E$-6$ & $-$ & 2.44E$-5$ & $-$  \\
    \hline
    80 & 2.70E$-7$ & 5.01 & 7.35E$-7$ & 5.05 \\
    \hline
    160 & 8.45E$-9$ & 5.00 & 2.31E$-8$ & 4.99 \\
    \hline
    320 & 2.64E$-10$ & 5.00 & 6.95E$-10$ & 5.06 \\
    \hline
    640 & 8.26E$-12$ & 5.00 & 2.13E$-11$ & 5.03 \\
    \hline
  \end{tabular}
  \caption{Lagrange extrapolation (without filter).}
  \label{lagrange}
\end{table}

From Table \ref{detec99} on, we add the last column with the
percentage of extrapolations for which no rejection, either in the 5
nodes or in the final result in the a posteriori criterion, has taken
place along the complete simulation.

\begin{table}[htb]
  \centering
  \begin{tabular}{|c|c|c|c|c|c|}
    \hline
    $n$ & Error $\|\cdot\|_1$ & Order $\|\cdot\|_1$ & Error
    $\|\cdot\|_{\infty}$ & Order $\|\cdot\|_{\infty}$ & $\%$
    Success \\
    \hline
    40 & 5.45E$-5$ & $-$ & 3.81E$-4$ & $-$ & 86.18 $\%$  \\
    \hline
    80 & 3.06E$-6$ & 4.15 & 3.65E$-5$ & 3.38 & 95.77 $\%$ \\
    \hline
    160 & 1.34E$-8$ & 7.83 & 2.10E$-7$ & 7.44 & 99.55 $\%$ \\
    \hline
    320 & 2.64E$-10$ & 5.67 & 6.95E$-10$ & 8.93 & 100.00 $\%$ \\
    \hline
    640 & 8.26E$-12$ & 5.00 & 2.13E$-11$ & 5.03 & 100.00 $\%$ \\
    \hline
  \end{tabular}
  \caption{Filter for detection of discontinuities, $\delta=\delta'=0.99$.}
\label{detec99}
\end{table}
\begin{table}[htb]
  \centering
  \begin{tabular}{|c|c|c|c|c|c|}
    \hline
    $n$ & Error $\|\cdot\|_1$ & Order $\|\cdot\|_1$ & Error
    $\|\cdot\|_{\infty}$ & Order $\|\cdot\|_{\infty}$ & $\%$
    Success \\
    \hline
    40 & 1.95E$-5$ & $-$ & 1.38E$-4$ & $-$ & 98.75 $\%$  \\
    \hline
    80 & 2.70E$-7$ & 6.17 & 7.35E$-7$ & 7.55 & 100.00 $\%$ \\
    \hline
    160 & 8.45E$-9$ & 5.00 & 2.31E$-8$ & 4.99 & 100.00 $\%$ \\
    \hline
    320 & 2.64E$-10$ & 5.00 & 6.95E$-10$ & 5.06 & 100.00 $\%$ \\
    \hline
    640 & 8.26E$-12$ & 5.00 & 2.13E$-11$ & 5.03 & 100.00 $\%$ \\
    \hline
  \end{tabular}
  \caption{Filter for detection of discontinuities, $\delta=0.9,$
    $\delta'=0.75$.}
\label{detec975}
\end{table}
\begin{table}[htb]
  \centering
  \begin{tabular}{|c|c|c|c|c|c|}
    \hline
    $n$ & Error $\|\cdot\|_1$ & Order $\|\cdot\|_1$ & Error
    $\|\cdot\|_{\infty}$ & Order $\|\cdot\|_{\infty}$ & $\%$
    Success \\
    \hline
    40 & 8.73E$-6$ & $-$ & 2.44E$-5$ & $-$ & 100.00 $\%$  \\
    \hline
    80 & 2.70E$-7$ & 5.01 & 7.35E$-7$ & 5.05 & 100.00 $\%$ \\
    \hline
    160 & 8.45E$-9$ & 5.00 & 2.31E$-8$ & 4.99 & 100.00 $\%$ \\
    \hline
    320 & 2.64E$-10$ & 5.00 & 6.95E$-10$ & 5.06 & 100.00 $\%$ \\
    \hline
    640 & 8.26E$-12$ & 5.00 & 2.13E$-11$ & 5.03 & 100.00 $\%$ \\
    \hline
  \end{tabular}
  \caption{Filter for detection of discontinuities, $\delta=0.75,$
    $\delta'=0.5$.}
\label{detec755}
\end{table}

From the results in those tables one can conclude that the detection behavior
improves with increasing resolution. The technical reason for this is
that the quotient
between the quantities appearing in \eqref{eq:Idelta} satisfies
\begin{equation*}
  \lim_{h_x\to 0}\frac{|u_i-u_{i_0}|+D(x_i)}{|v_i-u_{i_0}|+D(x_i)}=1.
\end{equation*}
Even at low resolutions, we observe that it is sufficient
to use a relatively restrictive threshold for not rejecting any point
in the extrapolations procedure at each time step.

{\bf Example 2.} We now perform a test omitting the $D(x_i)$ terms, which,
as stated in the previous section, help avoiding erratic node eliminations when
the differences are very close to be 0. The results can be seen at
Table \ref{erratic}, illustrating the importance of such terms.

\begin{table}[htb]
  \centering
  \begin{tabular}{|c|c|c|c|c|c|}
    \hline
    $n$ & Error $\|\cdot\|_1$ & Order $\|\cdot\|_1$ & Error
    $\|\cdot\|_{\infty}$ & Order $\|\cdot\|_{\infty}$ & $\%$
    Success \\
    \hline
    40 & 2.60E$-5$ & $-$ & 1.72E$-4$ & $-$ & 99.73 $\%$  \\
    \hline
    80 & 3.25E$-7$ & 6.32 & 1.60E$-6$ & 6.75 & 99.92 $\%$ \\
    \hline
    160 & 8.45E$-9$ & 5.27 & 2.31E$-8$ & 6.11 & 99.97 $\%$ \\
    \hline
    320 & 2.64E$-10$ & 5.00 & 6.95E$-10$ & 5.06 & 99.99 $\%$ \\
    \hline
    640 & 8.26E$-12$ & 5.00 & 2.13E$-11$ & 5.03 & 99.99 $\%$ \\
    \hline
  \end{tabular}
  \caption{Filter without $D(x_i)$ terms, $\delta=0.2,$ $\delta'=0.1$.}
\label{erratic}
\end{table}

We see that, indeed, without the $D(x_i)$ terms there are always some
nodes removed even using very low threshold values.

{\bf Example 3.} In order to illustrate the behaviour of our
  method in presence of small-cut cells, we now perform a test
  changing the location of the nodes by
  $x_j=-1+\left(j+\frac{1}{8}\right)h_x$ and following the procedure
  expounded in \eqref{eq:nodesinflow1d} and
  \eqref{eq:nodesoutflow1d}. For instance, to extrapolate 
  data to $x_*:=x_{-1}=-1-\frac{7}{8}h_x$, one first computes
  $v=N(x_*)-x_*=-1-x_{-1}=\frac{7}{8}h_x$, $C_x=E(v/h_x)=1$ and
  considers points $N'_q=x_{*}+qC_xh_x=x_{-1}+qh_x=x_{q-1}$.  Since
  there is a boundary condition at $N(x_*)=-1$, 
  $N'_1=x_0$ is not considered for extrapolation and   the selected five
  nodes are $\{-1, N'_2,\dots, N'_5\}=\{-1,x_1,\dots,x_4\}$

 The results obtained for $\delta=0.75,$ $\delta'=0.35$ are shown in Table \ref{nremoval}. 
 No node rejection occurred in this experiment.

  \begin{table}[htb]
  \centering
  \begin{tabular}{|c|c|c|c|c|}
    \hline
    $n$ & Error $\|\cdot\|_1$ & Order $\|\cdot\|_1$ & Error
    $\|\cdot\|_{\infty}$ & Order $\|\cdot\|_{\infty}$ \\
    \hline
    40 & 9.81E$-6$ & $-$ & 2.39E$-5$ & $-$  \\
    \hline
    80 & 3.06E$-7$ & 5.00 & 7.56E$-7$ & 4.98 \\
    \hline
    160 & 9.52E$-9$ & 5.00 & 2.28E$-8$ & 5.05 \\
    \hline
    320 & 2.97E$-10$ & 5.00 & 7.03E$-10$ & 5.02 \\
    \hline
    640 & 9.23E$-12$ & 5.01 & 2.12E$-11$ & 5.06 \\
    \hline
  \end{tabular}
  \caption{Lagrange extrapolation (node removal).}
  \label{nremoval}
\end{table}

Note that  in our numerical scheme and for accuracy reasons
  we have used $\Delta t=(h_x)^{\frac{5}{3}}$ and,
  therefore, no stability issue should appear  anyway for big enough $n$.
Forgetting about matching the spatial accuracy order with
  the time accuracy order, we
  set $n=80$, $\Delta t=0.9h_x$, that is, a CFL value of 0.9, and see that
our scheme is indeed stable and obtains good results as can be seen in
Figure \ref{fig:sremoval}.

\begin{figure}
  \centering
  \includegraphics[scale=0.4]{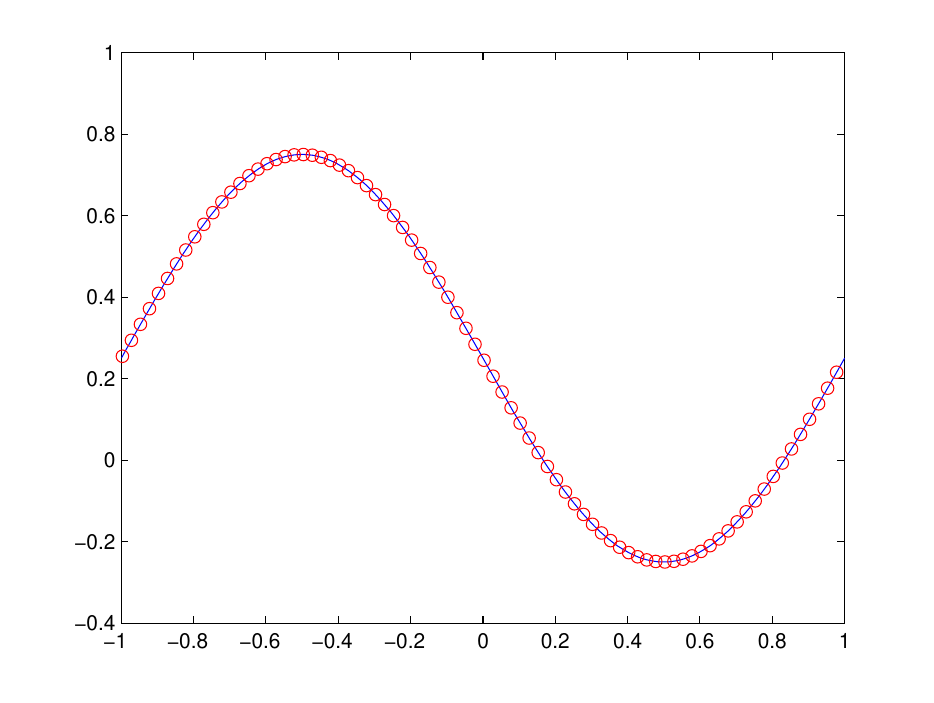}
  \caption{Stability.}
  \label{fig:sremoval}
\end{figure}

{\bf Example 4.} To complete the previous examples we now  analyze what happens
  if we attempt to extrapolate directly information
  at ghost cells  without the removal of nodes
  too close to the boundary, i.e., for $x_*=x_{-1}$ the stencil would
  be $\{-1,x_0,\dots,x_3\}$.
  For this experiment, we use
  the  grid from Example 3, a
  Courant number of $0.9$, i.e., $\Delta t=0.9 h_x$,
  and a five nodes extrapolation at both sides of the boundary as done in
  the previous experiments.
  The crucial difference with respect to Example 3 is
  that now we do not remove $N'_1$, thus
  resulting in a stability problem clearly visible
   already at the early stages of the simulation shown in
   Figure \ref{fig:oscillation_stabilized} (a),
   which ultimately lead to failure by numeric overflow.

  In order to illustrate that it is actually a CFL issue, we
    now repeat the simulation with a Courant number set again to $0.9$
    but based on the distance of the closest node of the inflow
  boundary to this last one (based on a spacing of
  $\frac{h_x}{8}$), i.e., $\Delta t=0.9 \frac{h_x}{8}$. In Figure \ref{fig:oscillation_stabilized} (b) it can be seen
  that now the scheme is stable.
  We conclude that the intermediate step consisting in extrapolating
  the information on nodes with adequate spacing is necessary in order to avoid
  unnecessarily severe time step restrictions.

  \begin{figure}[htb]
      \centering
      \begin{tabular}{cc}
      \includegraphics[width=0.4\textwidth]{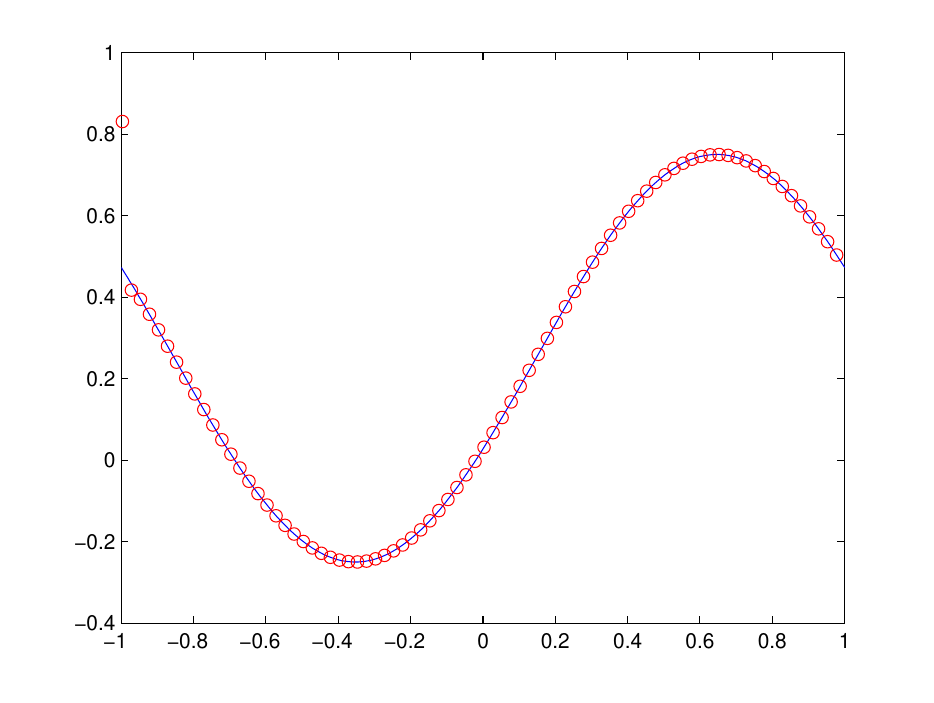} &
      \includegraphics[width=0.4\textwidth]{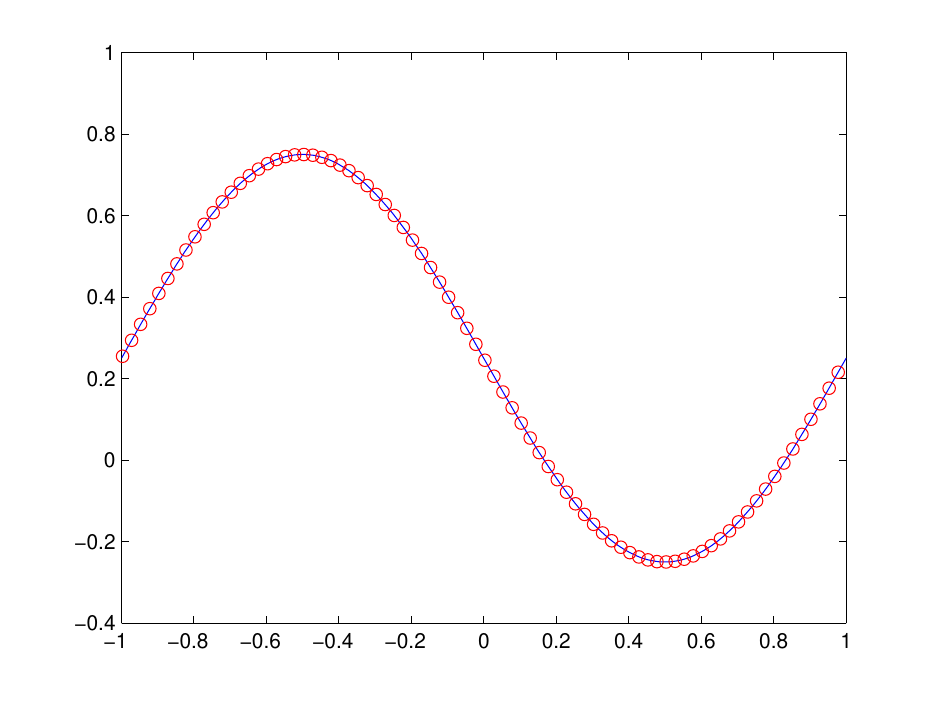}\\
      (a) & (b)
    \end{tabular}
    \caption{(a)$\Delta t=0.9h_x$, $t=0.147$. Oscillations appear;
      (b) $\Delta t=0.9\frac{h_x}{8}$, $t=1$. No oscillations}
\label{fig:oscillation_stabilized}
  \end{figure}

\subsubsection{Linear advection, discontinuous solution.}
We illustrate with this experiment the behavior of the schemes when
discontinuities are present and the entailed improvement with respect
to using Lagrange extrapolation with no filters. We consider the same
data as in the previous problem, but now the boundary condition is:
$$
u(-1,t)=g(t)=\left\{\begin{array}{ccc}
    0.25 & \textnormal{if} & t\leq1 \\
    -1 & \textnormal{if} & t>1 \\
    \end{array}\right.$$
  With this definition, the unique (weak) solution to this problem has
  a moving discontinuity and is given by:
$$u(x,t)=\left\{\begin{array}{ccc}
    -1 & \textnormal{if} & x<t-2 \\
    0.25 & \textnormal{if} & t-2\leq x\leq t-1 \\
    0.25+0.5\sin(\pi(x-t)) & \textnormal{if} & x\geq t-1 \\
    \end{array}\right.$$

In Figure \ref{comp} we check the graphical results that correspond to the simulation until
$t=1.5$, first using Lagrange extrapolation with no filters and
afterwards with a filter with  $\delta=0.75$ and $\delta'=0.5$, the
same values that have achieved no node rejections in the first
test. As it can be seen in Figure \ref{comp}, Lagrange extrapolation without
  filters leads to spurious oscillations around the left side of the
  discontinuity, while thresholding removes them.

\begin{figure}[htb]
\centering
\begin{tabular}{cc}
\includegraphics[width=0.38\textwidth]{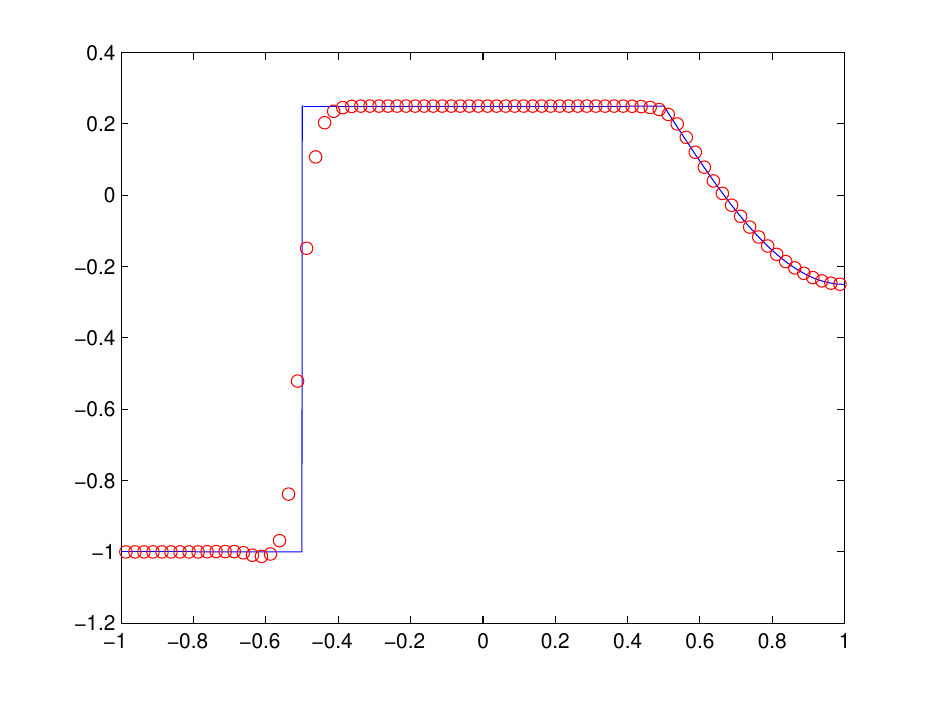}
&
\includegraphics[width=0.38\textwidth]{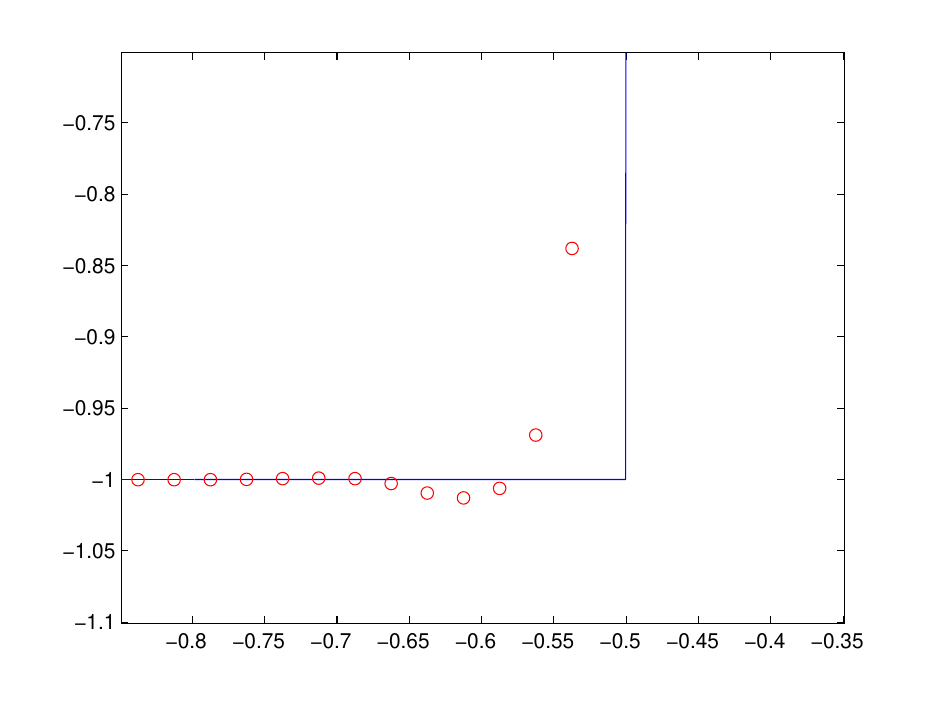}
\\
(a) Lagrange extrapolation. & (b) Lagrange extrapolation (zoom).\\
\includegraphics[width=0.38\textwidth]{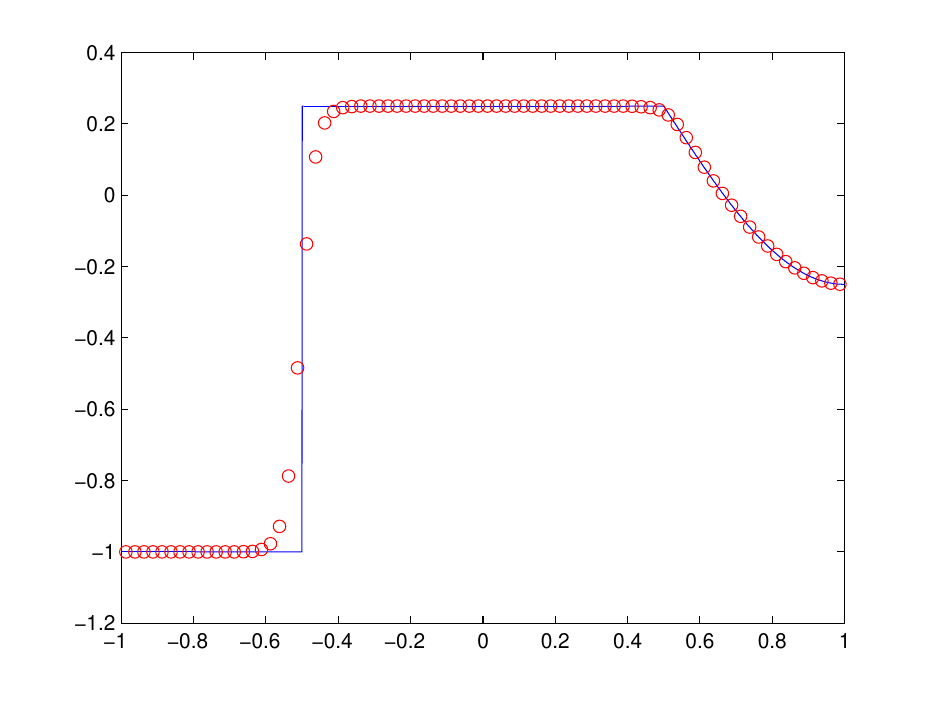}
&
\includegraphics[width=0.38\textwidth]{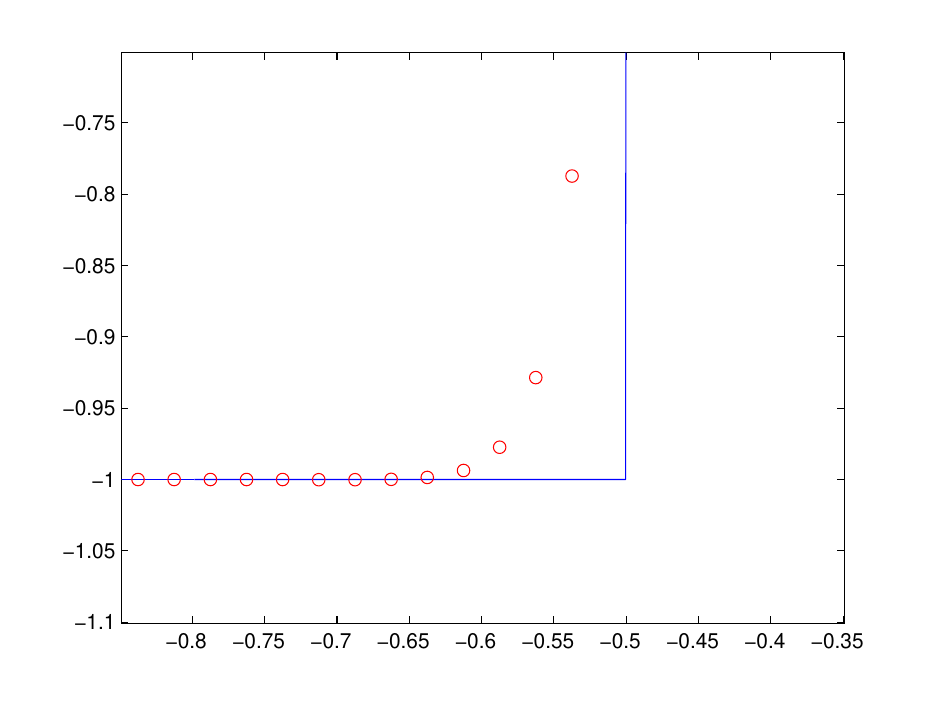}\\
(c) Extrapolation with thresholds. & (d) Extrapolation with thresholds
(zoom).
\end{tabular}
\caption{Comparison of different extrapolations for the linear
  advection test with discontinuous solution.}
\label{comp}
\end{figure}

\subsubsection{Burgers equation.}
Let's now perform some tests using Burgers equation
$$u_t+\left(\frac{u^2}{2}\right)_x=0,\quad\Omega=(-1,1),$$
with initial condition $u(x,0)=0.25+0.5\sin(\pi x)$ with outflow
condition at the right boundary and a left inflow
boundary conditions given by $u(-1,t)=g(t)$, where $g(t)=w(-1,t)$,
with $w$ the exact solution of the problem using periodic boundary
conditions.

For $t=0.3$ the solution is smooth and we get the following error
table for $n=40\cdot2^k,$ $0\leq k\leq 5$, and the same spacing
as the first test, using threshold values of $\delta=0.75$ and
$\delta'=0.5$, where no node rejection occurs at any resolution.

\begin{table}[htb]
  \centering
  \begin{tabular}{|c|c|c|c|c|}
    \hline
    $n$ & Error $\|\cdot\|_1$ & Order $\|\cdot\|_1$ & Error
    $\|\cdot\|_{\infty}$ & Order $\|\cdot\|_{\infty}$ \\
    \hline
    40 & 3.66E$-5$ & $-$ & 7.45E$-4$ & $-$  \\
    \hline
    80 & 6.96E$-7$ & 5.72 & 1.73E$-5$ & 5.43 \\
    \hline
    160 & 1.33E$-8$ & 5.70 & 3.58E$-7$ & 5.59 \\
    \hline
    320 & 3.34E$-10$ & 5.32 & 1.15E$-8$ & 4.96 \\
    \hline
    640 & 1.02E$-11$ & 5.04 & 3.43E$-10$ & 5.06 \\
    \hline
    1280 & 3.19E$-13$ & 4.99 & 1.03E$-11$ & 5.06 \\
    \hline
  \end{tabular}
  \caption{Error table for Burgers equation, $t=0.3$.}
  \label{burgers}
\end{table}

At $t=1.1$, a shock is fully developed in the interior of the
computational domain and enters the inflow boundary at $t=8$. At
$t=12$ it is located at $x=0$. We can see in Figure
\ref{burgersshock} that in this case
the discontinuities are well captured by our scheme as well.

\begin{figure}[htb]
\centering
\subfloat[$t=1.1$.]{\label{fb:1}\includegraphics[width=0.4\textwidth]{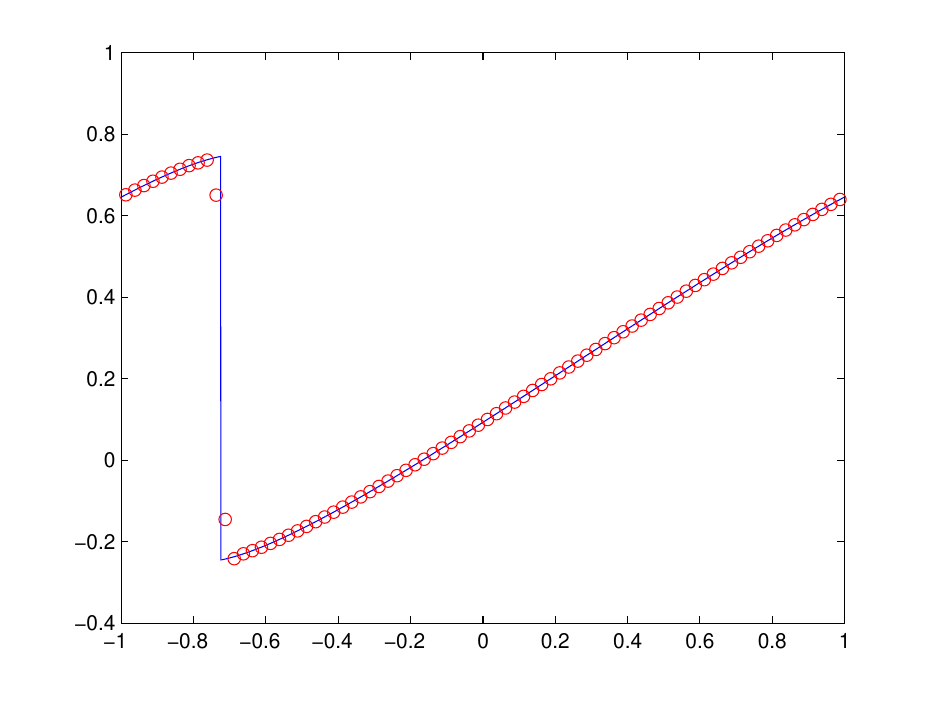}}
\subfloat[$t=12$.]{\label{fb:2}\includegraphics[width=0.4\textwidth]{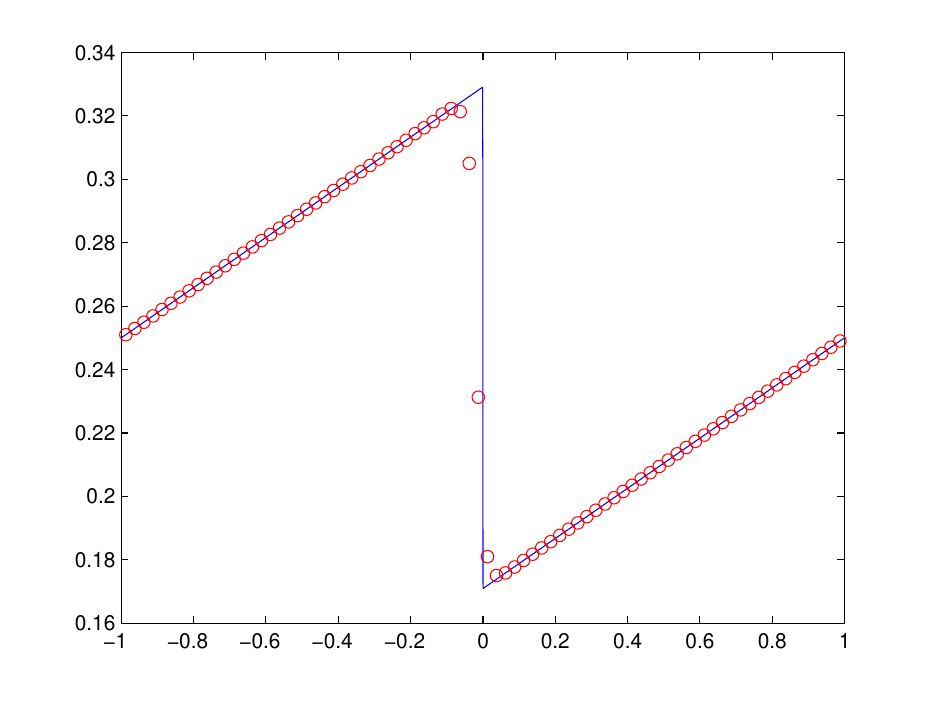}}
\caption{Shock in Burgers equation, $n=80$, $\delta=0.75$, $\delta'=0.5$.}
\label{burgersshock}
\end{figure}

\subsubsection{Euler equations.}

We end the one-dimensional experiments with an experiment using the Euler
equations

\begin{equation}\label{eq:eulereq1}
\begin{aligned}
&u_t+f(u)_x=0,\quad u=u(x,t),\quad\Omega=(0,1),\\
&u=\left[\begin{array}{c}
    \rho \\
    \rho v \\
    E \\
  \end{array}\right],\hspace{0.3cm}f(u)=\left[\begin{array}{c}
    \rho v \\
    p+\rho v^2 \\
    v(E+p) \\
  \end{array}\right],
\end{aligned}
\end{equation}
where $\rho$ is the density, $v$ is the
velocity and  $E$ is the specific energy of the system. The variable
$p$ stands for the pressure and is given by the equation of state:
$$p=\left(\gamma-1\right)\left(E-\frac{1}{2}\rho v^2\right),$$
where $\gamma$ is the adiabatic constant, that will be taken as
  $1.4$.

We simulate the interaction of two blast waves \cite{Colella} by using the
following initial data
$$u(x,0)=\left\{\begin{array}{ll}
    u_L & 0<x<0.1,\\
    u_M & 0.1<x<0.9,\\
    u_R & 0.9<x<1,
    \end{array}\right.$$
where $\rho_L=\rho_M=\rho_R=1$, $v_L=v_M=v_R=0$,
$p_L=10^3,p_M=10^{-2},p_R=10^2$. We set reflecting boundary conditions
at $x=0$ and $x=1$, simulating a solid wall at both sides. This
problem involves multiple reflections of shocks and rarefactions off
the walls and many interactions of waves inside the domain. We will
use the same extrapolation nodes setup as in the previous tests as well
as the same threshold values.

Figure \ref{eulershock} shows the density profile at $t=0.038$ at
two different resolutions, being the reference solution computed at
a resolution of $\Delta x=1/16000$. The figure clearly shows that the results are
satisfactory.

\begin{figure}[htb]
\centering
\subfloat[$\Delta x=1/800$.]{\label{fe:1}\includegraphics[width=0.4\textwidth]{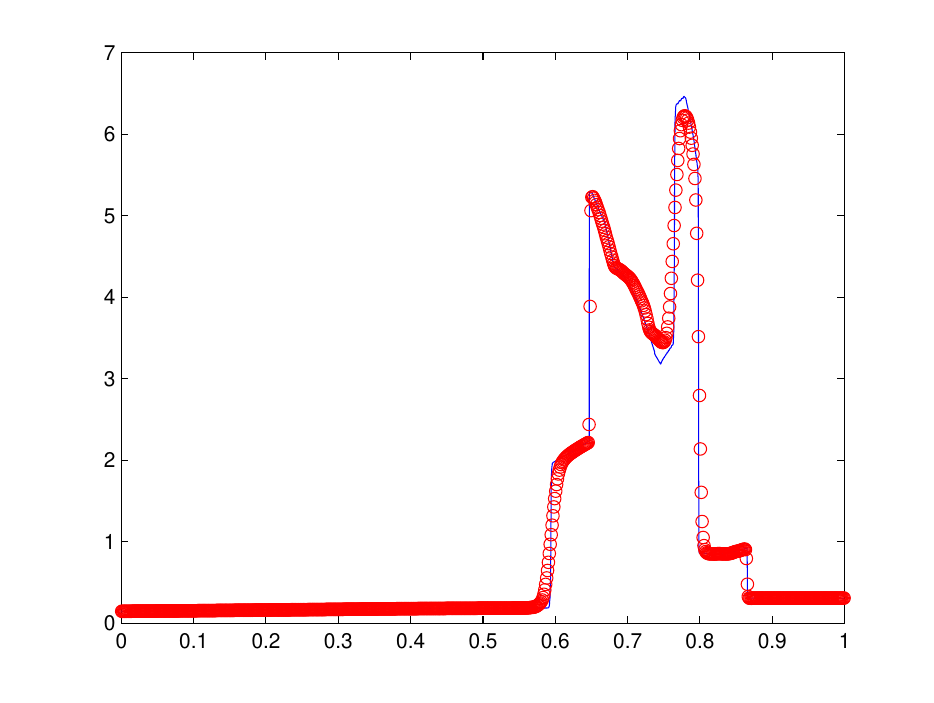}}
\subfloat[$\Delta x=1/1600$.]{\label{fe:2}\includegraphics[width=0.4\textwidth]{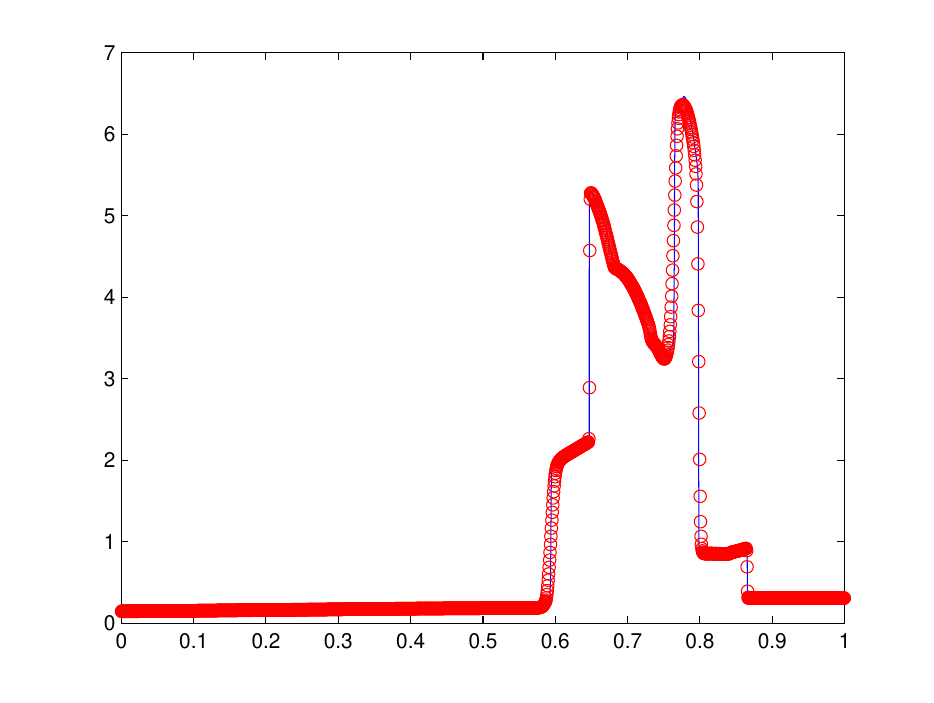}}
\caption{Density profile, $t=0.038$, $\delta=0.75$, $\delta'=0.5$.}
\label{eulershock}
\end{figure}

\subsection{Two-dimensional experiments}
The equations that will be considered in this section are the
two-dimensional Euler equations for inviscid gas dynamics
\begin{equation}\label{eq:eulereq}
\begin{aligned}
&u_t+f(u)_x+g(u)_y=0,\quad u=u(x,y,t),\\
&u=\left[\begin{array}{c}
    \rho \\
    \rho v^x \\
    \rho v^y \\
    E \\
  \end{array}\right],\hspace{0.3cm}f(u)=\left[\begin{array}{c}
    \rho v^x \\
    p+\rho (v^x)^2 \\
    \rho v^xv^y \\
    v^x(E+p) \\
  \end{array}\right],\hspace{0.3cm}g(u)=\left[\begin{array}{c}
    \rho v^y \\
    \rho v^xv^y \\
    p+\rho (v^y)^2 \\
    v^y(E+p) \\
  \end{array}\right].
\end{aligned}
\end{equation}
In these equations,  $\rho$ is the density, $(v^x, v^y)$  is the
velocity and  $E$ is the specific energy of the system. The variable
$p$  stands for the pressure and is given by the equation of state:
  $$p=(\gamma-1)\left(E-\frac{1}{2}\rho((v^x)^2+(v^y)^2)\right),$$
  where $\gamma$ is the adiabatic constant, that will be taken as
  $1.4$ in all the experiments.

\subsubsection{Double Mach Reflection}
This experiment uses the Euler equations to model a vertical right-going Mach
10 shock colliding with an equilateral triangle. By symmetry, this is
equivalent to a collision with a ramp with a slope of 30 degrees with
respect to the horizontal line, which is how we will model the
simulation to half the computational cost.

The data for this problem are the following:

$$\Omega=\left\{(x,y)\in(0,4)\times(0,4):
\hspace{0.1cm}y>\frac{\sqrt{3}}{3}\left(x-\frac{1}{4}\right)\right\}.$$
 The domain with the corresponding boundary conditions is sketched in
 Figure \ref{fig:sketch}

\begin{figure}[htb]
  \centering
  \includegraphics[width=0.4\textwidth]{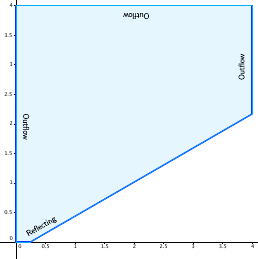}
  \caption{Domain for the double Mach reflection test.}
  \label{fig:sketch}
\end{figure}
The initial conditions are the following:
$$\begin{array}{rclcc}
  u=(\rho,v^x,v^y,E)&=&(8.0,8.25,0,563.5)&\textnormal{if}&x\leq\frac{1}{4}\\
  u=(\rho,v^x,v^y,E)&=&(1.4,0,0,2.5)&\textnormal{if}&x>\frac{1}{4}\\
\end{array}
$$

The most commonly used strategy for this simulation (see \cite{Colella})
 is to rotate the reference frame by $-30$ degrees, so that
 the simulation is cast into a rectangular domain, with a shock that
 is inclined 60
degrees with respect to the horizontal. In our case, we perform the
simulation with the original problem to see that the improvement
achieved by increasing the order of the extrapolations at the boundary
leads to results that are comparable to those obtained with the
rotated version.

Following the notation of the previous section, we have selected the
values $R=10$, $M=3$ for the boundary (substencils are of the same size as
those in WENO5). The reason for selecting  $R=10$ is not achieving an
order higher than the one of the method, but having a wider stencil with more
room for a safe selection of a substencil in smoothness regions.

We perform the simulation until  $t=0.2$. The experiment
consists in different simulations with different threshold values,
considering also a version with a unique point in the stencil (order
1). In Figure \ref{originalDMR} we present the result for the density
$\rho$ at a resolution of $h_x=h_y=\frac{\sqrt{3}}{2}\frac{1}{640}$,
which is equivalent to a resolution
$\hat{h}_x=\hat{h}_y=\frac{1}{640}$ in the rotated experiment. A comparison of the results for the original and the rotated experiment for different extrapolation options is shown in Figure \ref{ampl}.

\begin{figure}[htb]
  \centering
    \subfloat{\includegraphics[width=0.6\textwidth]{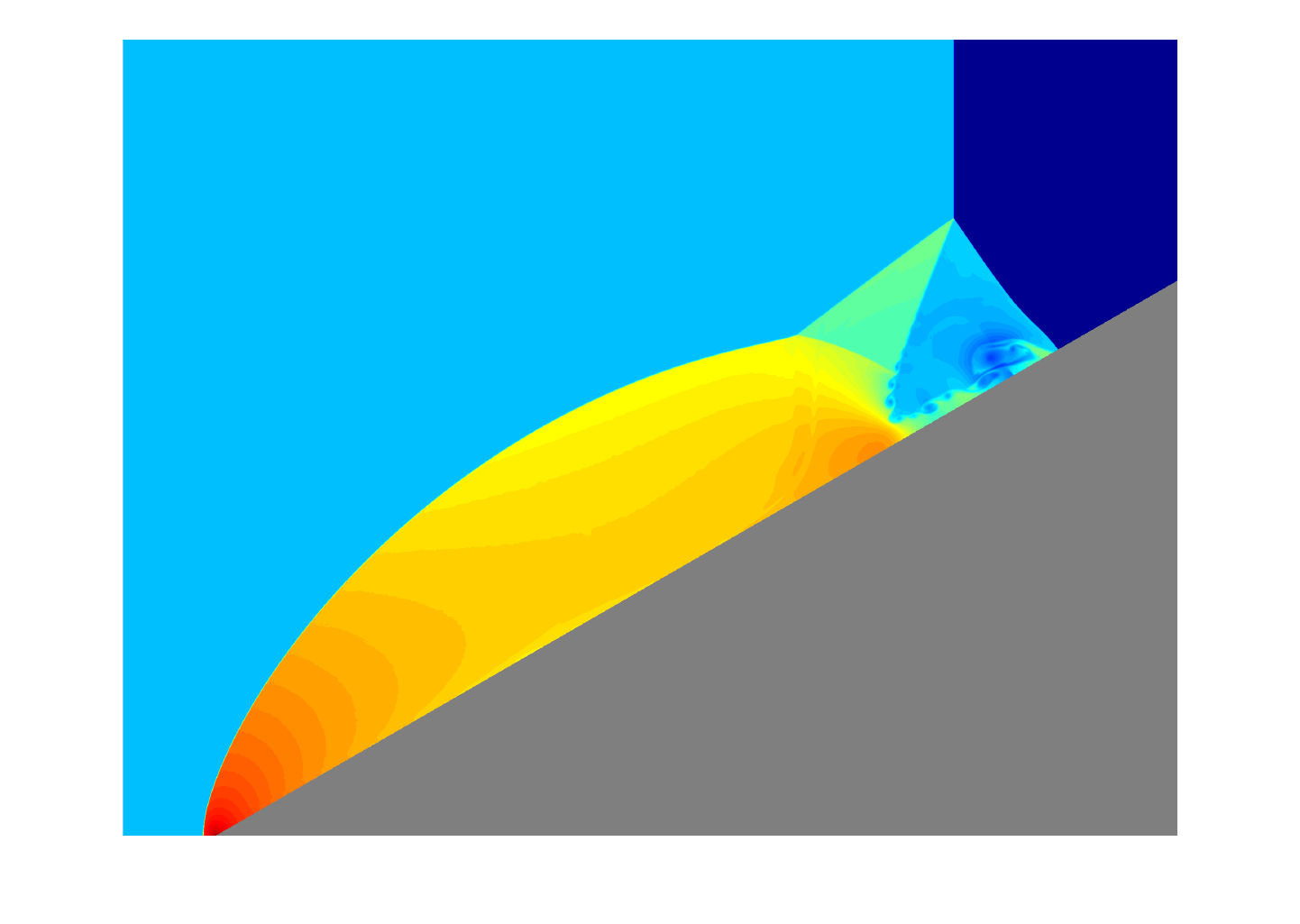}}
    \subfloat{\includegraphics[width=0.08\textwidth]{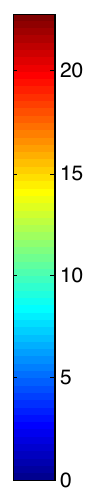}}
    \caption{Original problem.}
    \label{originalDMR}
\end{figure}
\begin{figure}[htb]
\centering
\subfloat[Rotated domain.]{\label{figg:1}\includegraphics[width=0.45\textwidth]{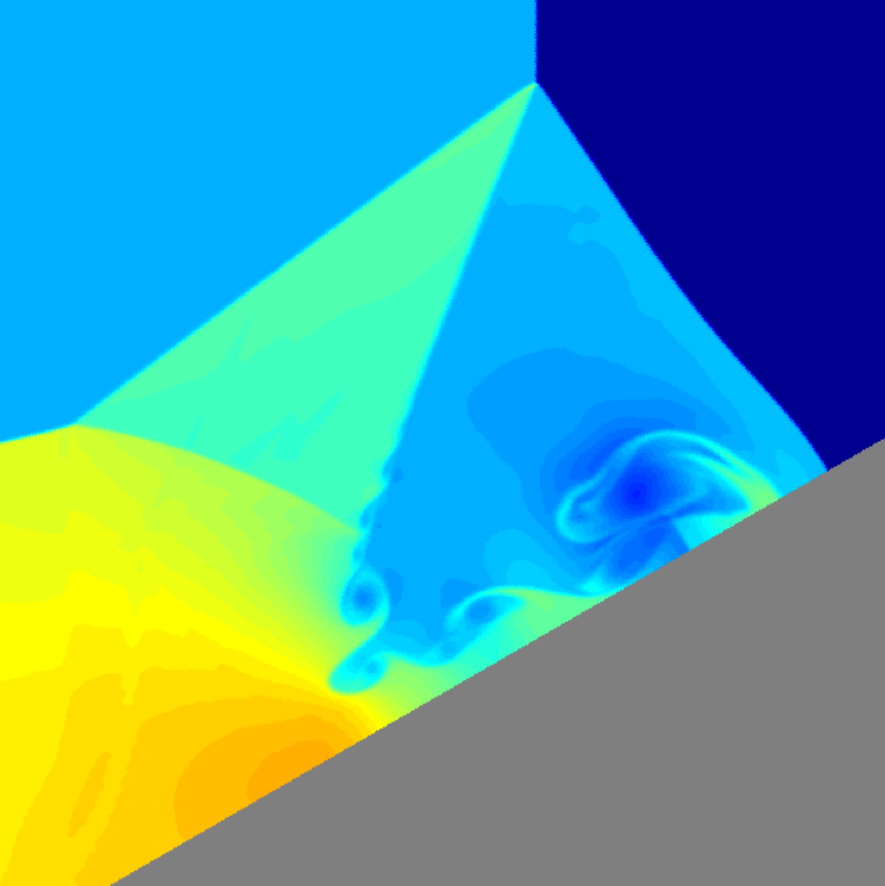}}
\subfloat[1st o. extrapolations.]{\label{figg:2}\includegraphics[width=0.45\textwidth]{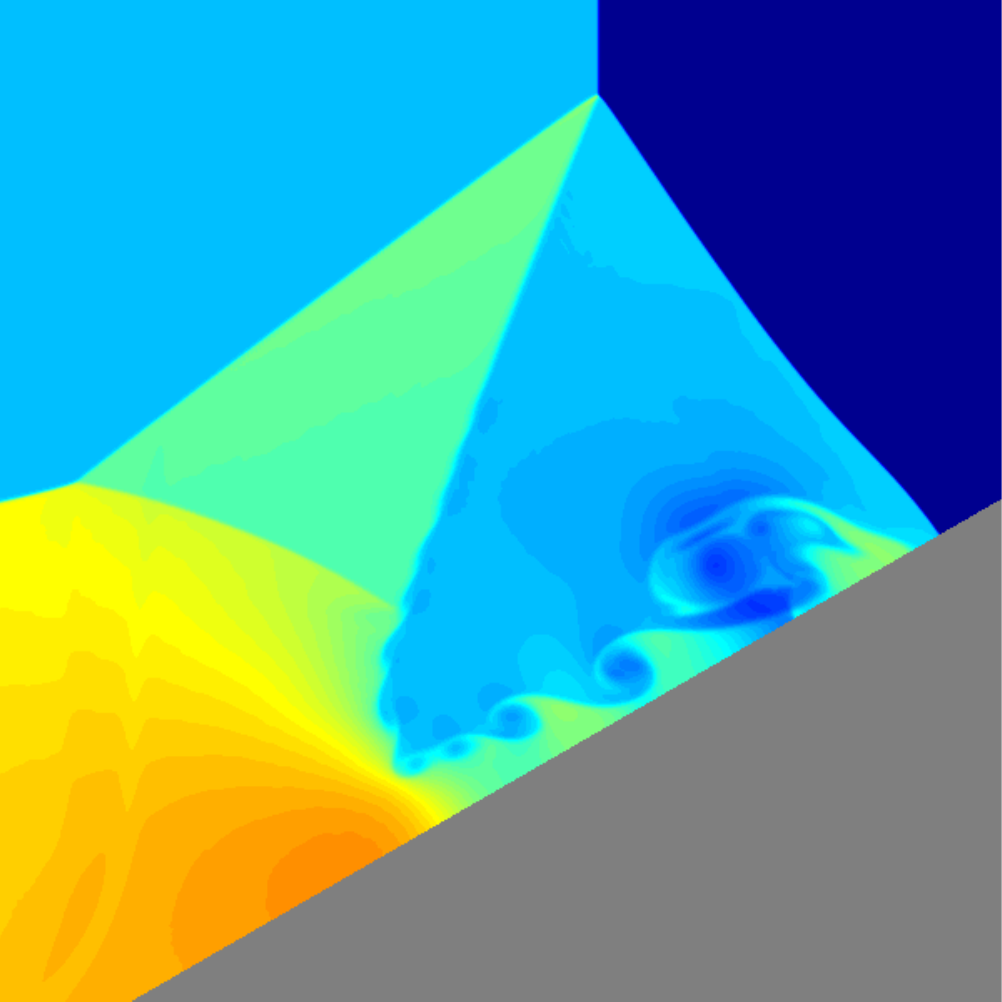}}\\
\subfloat[5th o. extrapolations. $\delta=\delta'=0.9.$]{\label{figg:3}\includegraphics[width=0.45\textwidth]{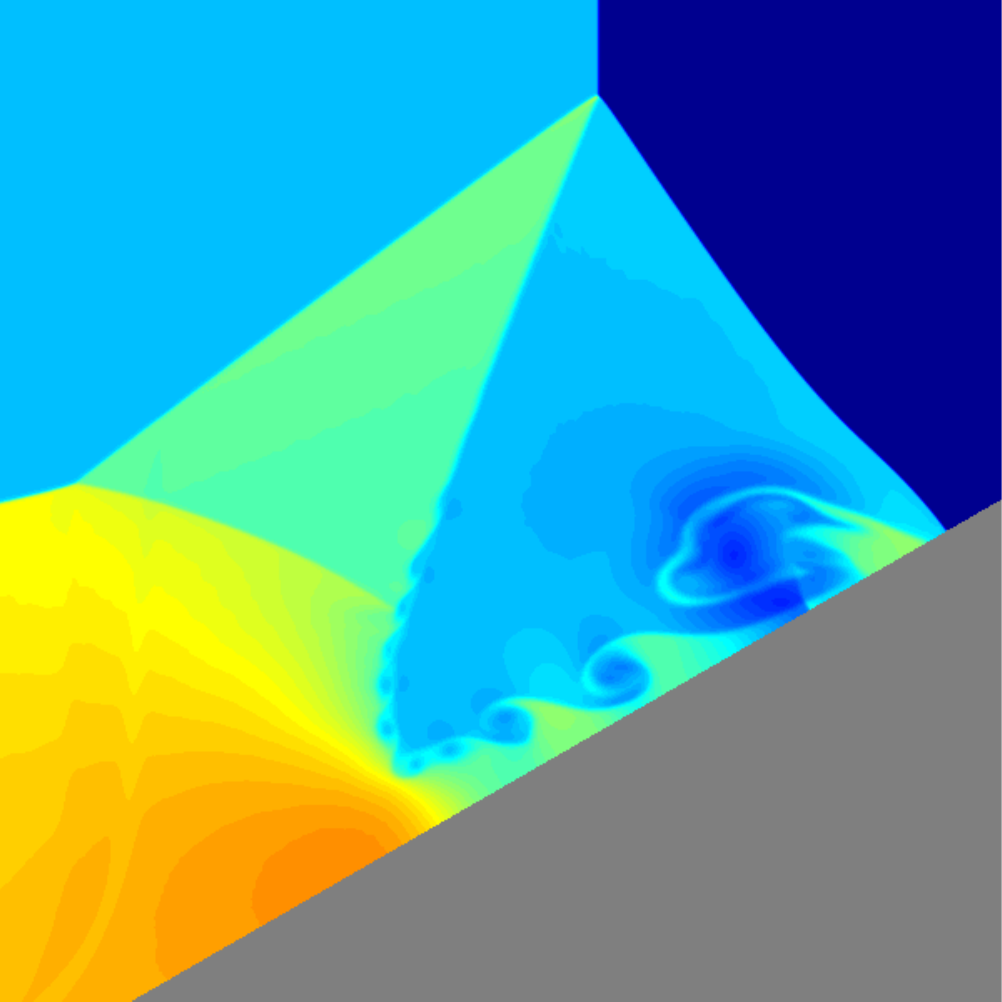}}
\subfloat[5th o. extrapolations. $\delta=\delta'=0.5.$]{\label{figg:4}\includegraphics[width=0.45\textwidth]{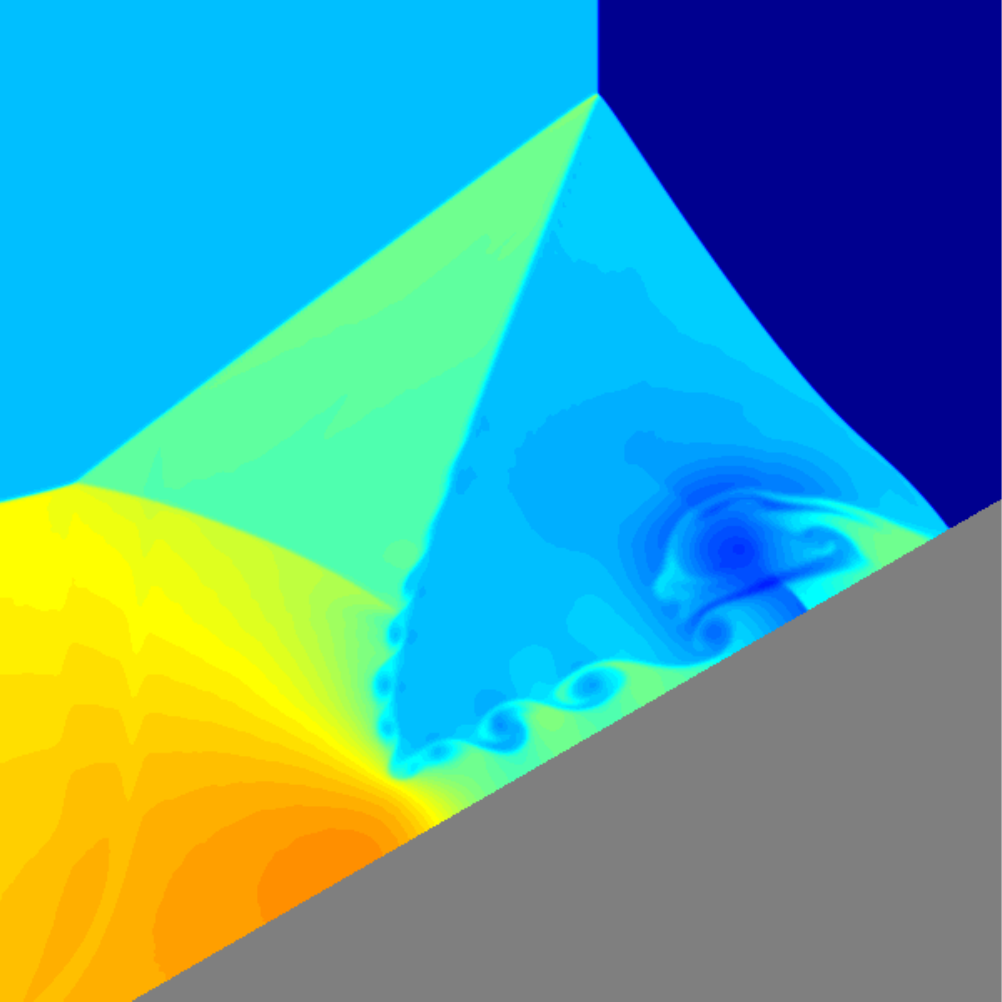}}\\
\subfloat[5th o. extrapolations. $\delta=\delta'=0.35.$]{\label{figg:5}\includegraphics[width=0.45\textwidth]{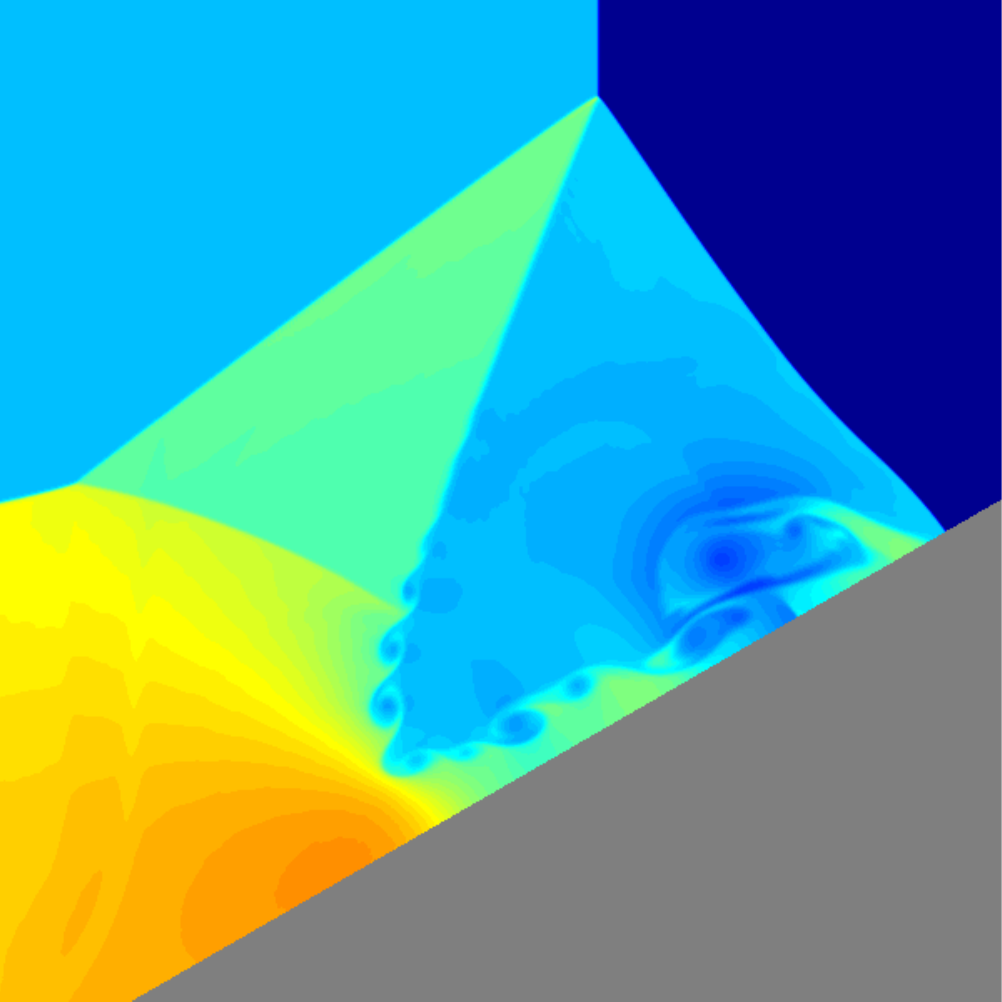}}
\subfloat[5th o. extrapolations. $\delta=\delta'=0.2.$]{\label{figg:6}\includegraphics[width=0.45\textwidth]{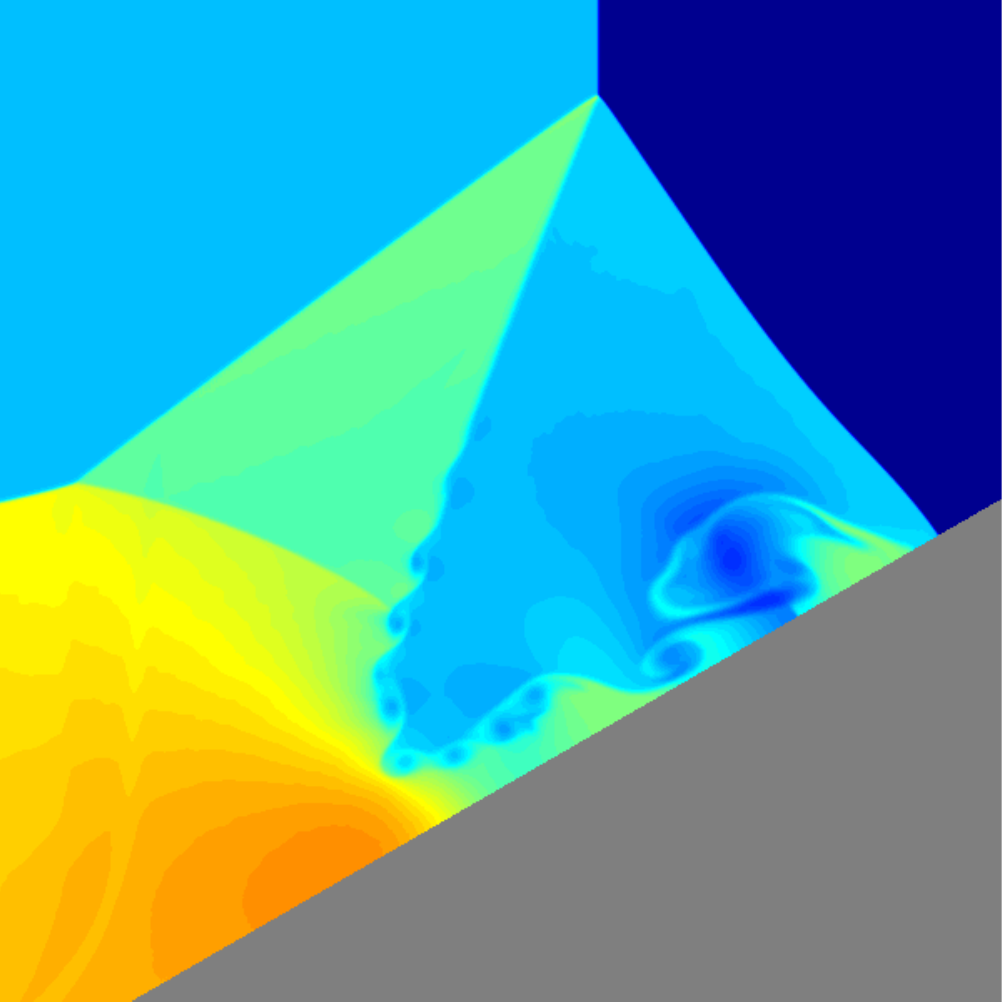}}
\caption{Enlarged view of the turbulence zone.}
\label{ampl}
\end{figure}
\begin{figure}[htb]
\centering
\subfloat[Rotated domain.]{\label{figgg:1}\includegraphics[width=0.45\textwidth]{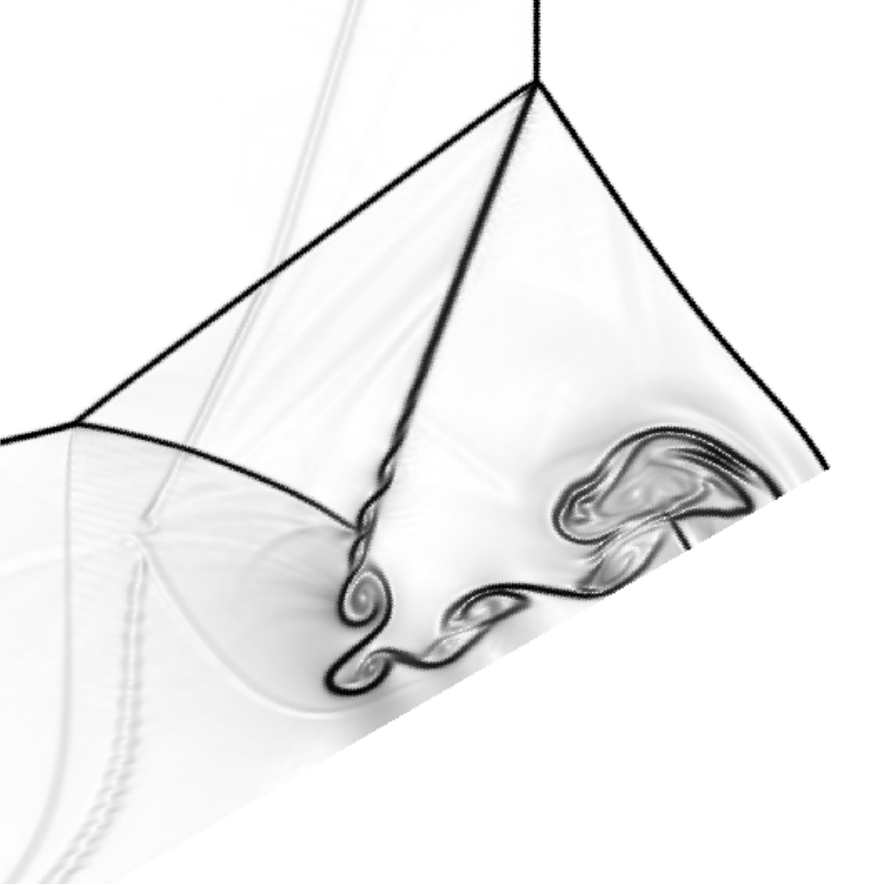}}
\subfloat[1st o. extrapolations.]{\label{figgg:2}\includegraphics[width=0.45\textwidth]{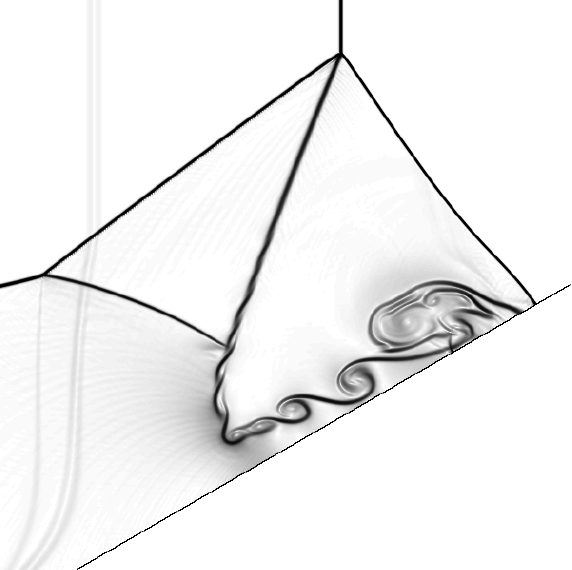}}\\
\subfloat[5th o. extrapolations. $\delta=\delta'=0.9.$]{\label{figgg:3}\includegraphics[width=0.45\textwidth]{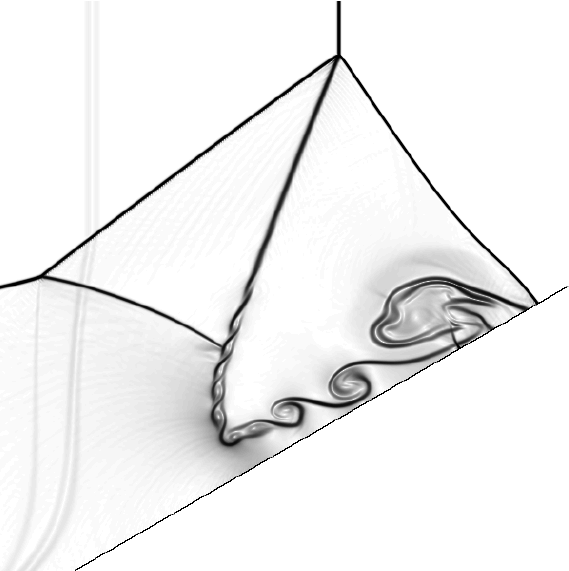}}
\subfloat[5th o. extrapolations. $\delta=\delta'=0.5.$]{\label{figgg:4}\includegraphics[width=0.45\textwidth]{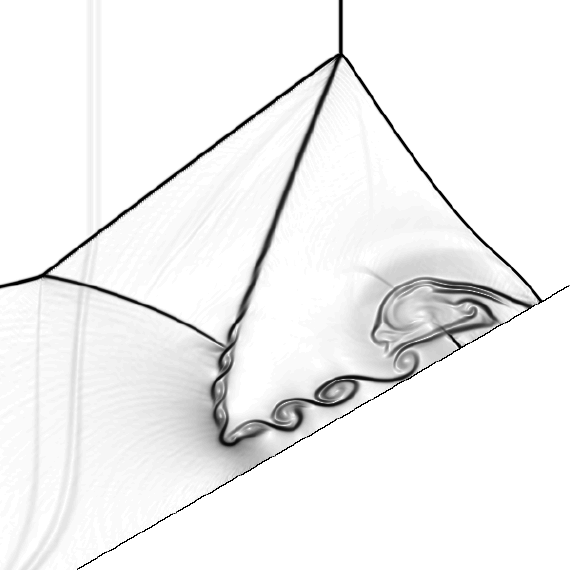}}\\
\subfloat[5th o. extrapolations. $\delta=\delta'=0.35.$]{\label{figgg:5}\includegraphics[width=0.45\textwidth]{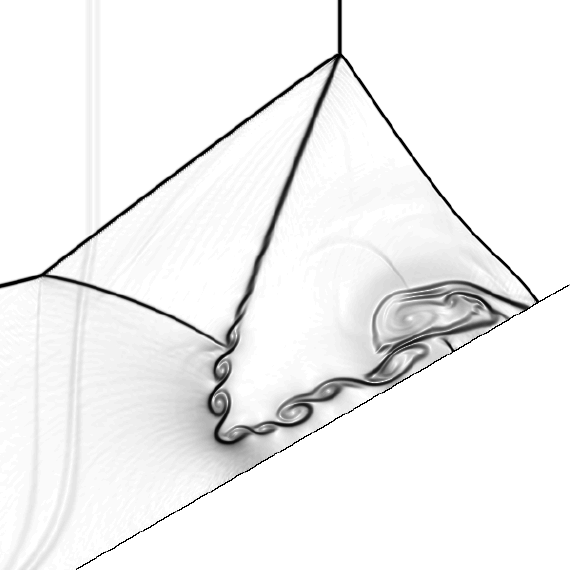}}
\subfloat[5th o. extrapolations. $\delta=\delta'=0.2.$]{\label{figgg:6}\includegraphics[width=0.45\textwidth]{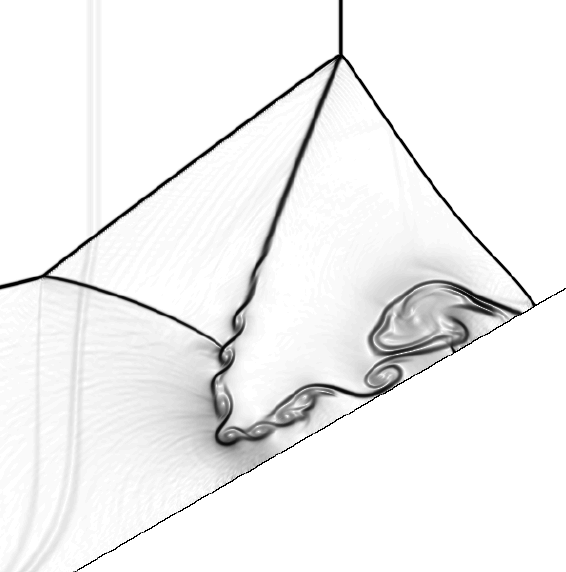}}
\caption{Enlarged view of the turbulence zone (Schlieren).}
\label{amplsch}
\end{figure}
  The Schlieren plots shown in Figures \ref{amplsch}, \ref{circle} and \ref{circles} display the
  gradients of the density field in an exponential scale in a
  gray scale, where darker
  tonalities correspond to higher values (see \cite{MarquinaMulet03}
  and references therein for details).

  As it can be seen, lower threshold values lead to better defined
  vortices. Also, note that according to the results on the
    figures, the rotated problem (with first order extrapolations)
    looks better than the original problem using also first order
    extrapolations. One of the reasons might be the fact that the cell
    centers are exactly located on each normal line, leading to exact
    values on the first step of the extrapolation process (we recall
    the reader that this step consists in extrapolating information to
    points on normal lines from the values of the computational domain).

\subsubsection{Interaction of a shock with a circular obstacle}

We now change our data to a right-going vertical Mach 3 shock initially located at $x=0.1$ with a
circular obstacle with center $(0.5,1)$ and radius 0.2  into a square domain
$(0,2)\times(0,2)$. This experiment has already been performed in
\cite{Boiron} using penalization techniques.
\begin{figure}[hbpt]
      \centering
      \begin{tabular}{cc}
\hspace{-5pt}        \includegraphics[height=6cm]{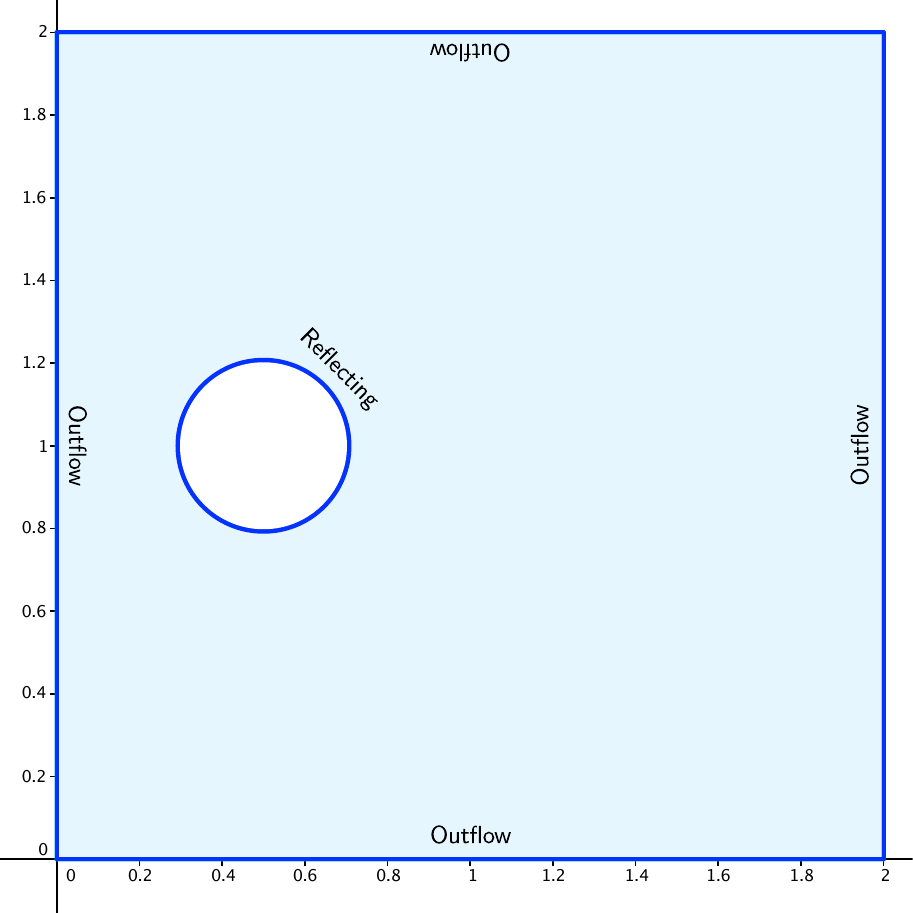} &\hspace{-10pt}\includegraphics[height=6cm]{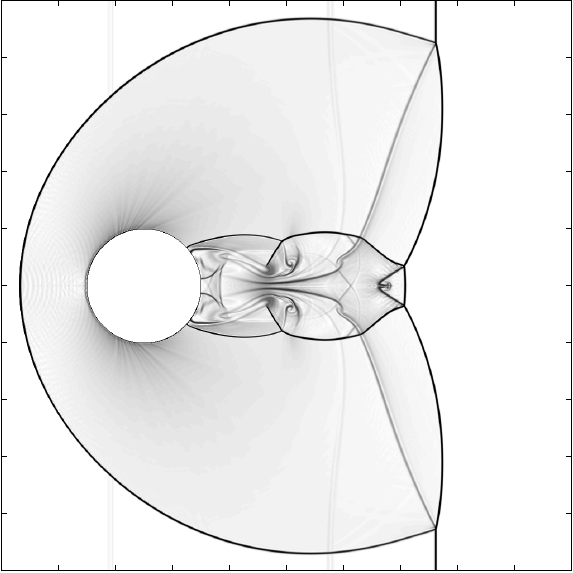}\\
      (a) & (b)
    \end{tabular}
      \caption{Circle reflection test: (a) domain; (b) simulation for $t=0.4$.}
      \label{circle}
\end{figure}
To halve the computational time by exploiting the symmetry of the
solution,
we run a simulation until $t=0.4$ and a mesh size of
$h_x=h_y=\frac{1}{512}$ on the upper half of the
domain, by adding reflecting boundary conditions at the bottom. A Schlieren
plot of the result can be seen at Figure \ref{circle}.
As it can be seen, the results are very similar to those obtained in
\cite{Boiron}.

\subsubsection{Interaction of a shock with multiple circular obstacles}
We repeat the previous experiment by adding multiple circles in the
domain as shown in Figure \ref{circles}. This test can also be found in
\cite{Boiron}.
In this case, we run
the simulation until $t=0.5$ and a mesh size of
$h_x=h_y=\frac{1}{512}$ on the whole domain.
As in the previous experiment, we present a Schlieren plot for the last
time step in Figure \ref{circles}.
These results are again consistent   with those obtained in \cite{Boiron}.
\begin{figure}[hbpt]
      \centering
      \begin{tabular}{cc}
\hspace{-5pt}        \includegraphics[height=6cm]{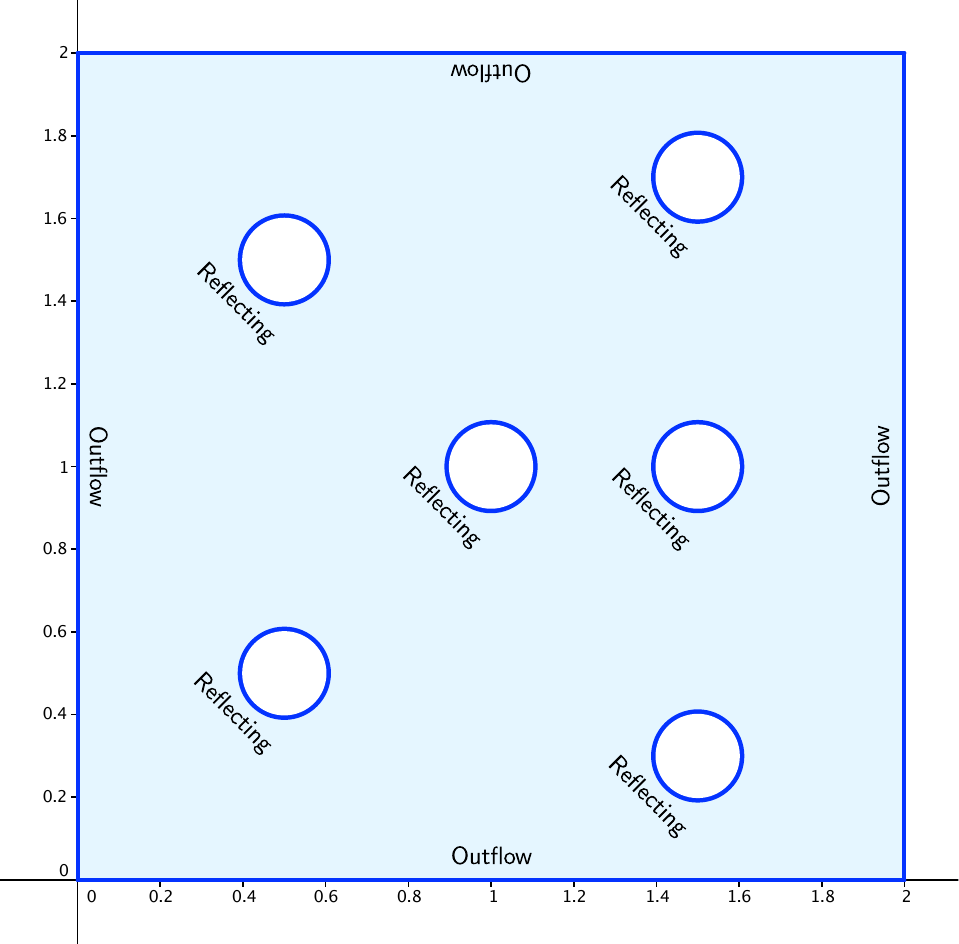}&\hspace{-10pt}\includegraphics[height=6cm]{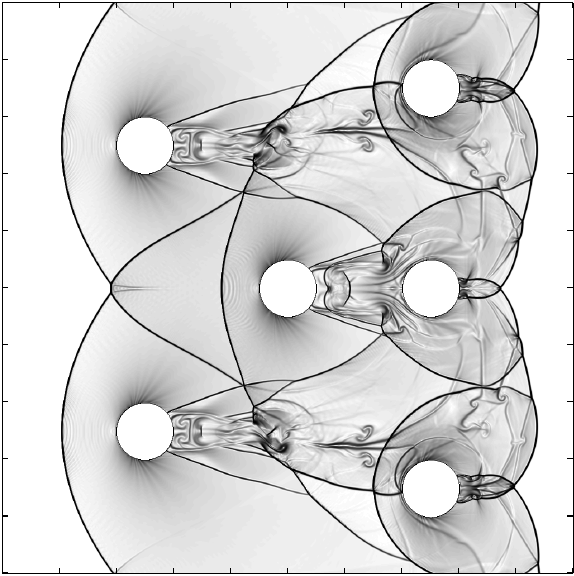}\\
      (a) & (b)
    \end{tabular}
      \caption{Circles reflection test: (a) domain; (b) simulation for $t=0.5$.}
      \label{circles}
\end{figure}

\section{Conclusions and future work}\label{scn}
In this paper we have presented some techniques for data extrapolation  to
handle boundary conditions for finite difference numerical methods for
hyperbolic conservation laws. We have obtained some  successful simulations
in non rectangular domains. This illustrates that Lagrange
extrapolation is a viable technique for filling-in auxiliary data at
the ghost cells, as long as sufficient care is taken for accounting
for possible discontinuities.

Furthermore, these techniques are designed to avoid an order loss at
the boundary of complex domains in methods that require Cartesian
meshes, a loss that can propagate to the rest of the data, thus
notably decreasing the simulation quality. The results that have been
obtained with these techniques entail an improvement that solves the
previous problem without a significant increase in computational
time at not excessively low resolutions.

The extrapolation techniques proposed in this paper have the advantage
of letting a regulation of the tolerance to some variations at the
boundary by using a threshold parameter. This, besides being data
scale independent,  is useful in simulations
with strong turbulence or, in general, with wide regions where the data
is not smooth.
However, the need of tuning the thresholding parameters to the
  particular problem represents a drawback of the method.

Therefore, as future work we encompass studying the design of
weighting methods, akin to WENO, that let performing extrapolations
taking advantage of the idea that has been used for extrapolation
based on a threshold parameter, by assigning a convenient weight in
each case, without needing a Boolean choice as in the case of using
thresholds.

We are also regarding, on the other hand, the implementation of this
methodology of boundary extrapolation techniques to Adaptive Mesh
Refinement algorithms \cite{BaezaMulet2006,BergerColella1989}, trying not to excessively compromise their
recognized efficiency on rectangular domains.

\appendix
\section{Computations of intersections and normals for
  meshing}\label{ap:mesh}
We have implemented a method to automatically compute  the interior
cells to $\Omega$, the ghost cells and the normal (to $\partial
\Omega$) vectors associated to them. We require a parametric
definition of the curve described by  $\partial\Omega$.

The knowledge of the Lipschitz constants of this curve is enough
for computing the intersections of the boundary with the mesh lines.
Indeed, with the notation
$\alpha=(\alpha_x,\alpha_y):[a,b]\rightarrow\mathbb{R}^2$ for the curve
described by
$\partial\Omega$, if $L_x$ and $L_y$ are upper bounds
for the Lipschitz constants of the respective components of $\alpha$,
then the values
$\Delta_xs:=\frac{h_x}{L_x}, \Delta_ys:=\frac{h_y}{L_y}$
satisfy
$$|\alpha_x(s+\Delta_xs)-\alpha_x(s)|\leq
h_x\quad \forall s\in[a,b-\Delta_xs]$$
$$|\alpha_y(s+\Delta_ys)-\alpha_y(s)|\leq
h_y\quad \forall s\in[a,b-\Delta_ys].$$
If we detect that a vertical or horizontal mesh line passes through
two values of the parameter, we use Newton-Raphson's method, combined
with bisection to safely approximate the intersection.

A similar strategy is used for the computations of the normals. In
this case, given a ghost cell center, $(x,y)$, we start from the
parameter value of the intersection that has generated that cell,
$s_0$, for this should not be far from the value of the parameter that
corresponds to the point $N(x, y)\in\partial \Omega$ that determines
the normal vector to the boundary passing by $(x, y)$.

Our search criterion is based on the sign of the scalar product of the
tangent vector, $(\alpha'_x(s), \alpha'_y(s))$ and the vector formed
by  $(\alpha_x(s), \alpha_y(s))$ and  $(x,y)$. At any rate, we aim to
find  $s$ satisfying the equation
\begin{equation}
  \label{eq:normal}
  f(s):=\langle(\alpha'_x(s),\hspace{0.1cm}\alpha'_y(s)),\hspace{0.1cm}(\alpha_x(s)-x,\hspace{0.1cm}
\alpha_y(s)-y)\rangle=0.
\end{equation}
To accomplish this, we start from  $s=s_0$ and evaluate this
expression at points at each side of $s_0$ until a sign change in $f$
is found. At this point we can use Newton-Raphson's method, together
with bisection, to safely approximate the solution of
\eqref{eq:normal} to the required precision.

\section{Linear stability analysis for inflow boundary
  conditions}
\label{ap:stability}
Since a GKS stability analysis (see \cite{GKS72}) is out of the scope of
this paper,
we perform a simpler stability analysis to illustrate the strong
influence of the spacing of the closest node to the boundary in the
stability of the global numerical method.

For the sake of the exposition, we consider a linear advection equation
$$u_x+u_t=0$$
defined on $\Omega=(0,1)$ and inflow condition at $x=0$, $u(0,
t)=g(t)$. Let $h_x>0$, $\Delta t>0$,
$k\in(0,1]$ and define our grid points as $x_j:=(k+j)h_x,$
$0\leq j\leq E(\frac{1}{h_x})$ and $t_n:=n\Delta t$.

For solving numerically the equation, we use a first order upwind
scheme:
$$U_j^{n+1}=U_j^n-\frac{\Delta t}{h_x}(U_j^n-U_{j-1}^n)$$
Assume we want to use linear interpolation at the
boundary. This involves the inflow condition itself and the node
$U_0^n$ at each time step $n$. Naming $G^n:=g(t_n)$, if we want to
extrapolate at $x_{-1}=(k-1)h_x$, where no information is available, at time
step $n$ using these two points, it can be shown by performing
a linear extrapolation that
$$U_{-1}^n=\frac{1}{k}G^n+\frac{k-1}{k}U_0^n.$$

Since we want to focus on a stability analysis around the left boundary, where
inflow conditions are prescribed, the scheme at $x_0$ reads

$$U_0^{n+1}=U_0^n-\frac{\Delta t}{h_x}(U_0^n-U_{-1}^n)=
U_0^n-\frac{\Delta
  t}{h_x}\left(U_0^n-\frac{1}{k}G^n+\frac{1-k}{k}U_0^n\right)=$$
$$=\left(1-\frac{\Delta t}{kh_x}\right)U_0^n+\frac{\Delta
  t}{kh_x}G^n=U_0^n-\frac{\Delta t}{kh_x}\left(U_0^n-G^n\right).$$

With this, following a standard Von Neumann stability analysis from this
  new scheme, a necessary condition for stability is
$$0\leq\frac{\Delta t}{kh_x}\leq 1,$$
which is equivalent to
$$k\geq\frac{\Delta t}{h_x}=\text{CFL}.$$

\section{Accuracy of schemes obtained from extrapolation at ghost
  cells}
\label{ap:accuracy}
Related to Section \ref{ss:fdweno}, let us denote
\begin{align*}
  \Phi(a_{-p-1},\dots,a_q)&=\hat{f}(a_{-p},\dots,a_{q})-\hat{f}(a_{-p-1},\dots,a_{q-1})\\
  U(h)&=(u(x-(p+1)h),\dots,u(x+qh))
\end{align*}
for fixed $x$ and sufficiently smooth $u$, $\hat f$. Then
\eqref{eq:36}  is equivalent to
\begin{equation}\label{eq:39}
  F(h)=\Phi(U(h))=h f(u(x))_x + \bigO(h^{r+1}),
\end{equation}
which is in turn equivalent to
\begin{equation}\label{eq:38}
  F'(0)=f'(u(x))u'(x),\quad F^{(n)}(0)=0,\quad n=2,\dots,r.
\end{equation}
If $\hat f$ were a linear function, these relations would immediately
yield a linear system of equations for its coefficients. In general, a
more complicated formula equivalent to \eqref{eq:38} can be obtained
as follows:
By induction on $n$ the following vectorial version of Fa\`a di
Bruno's Formula \cite{faadibruno1857}  can be proved:
\begin{align}\label{eq:371}
  F^{(n)}(h)=\sum_{m=1}^{n}\sum_{p_1,\dots,p_m=1}^{n-m+1}
  \alpha_{p_1,\dots,p_m}
  \sum_{j_1,\dots,j_m=-p-1}^{q} \frac{\partial^{m}\Phi(U(h))}{\partial
    a_{j_1}\dots \partial a_{j_m}} U_{j_1}^{(p_1)}(h)\dots
  U_{j_m}^{(p_m)}(h),
\end{align}
for suitable and uniquely defined coefficients
$\alpha_{p_1,\dots,p_m}$. From \eqref{eq:37} and the fact that
$U_j^{(\nu)}=j^\nu u^{(\nu)}(x)$, for $j=-p-1,\dots,q$ and $\nu\geq
0$, we deduce that
\begin{equation}\label{eq:37}
  \begin{aligned}
    &F^{(n)}(0)=\sum_{m=1}^{n}\sum_{p_l=1}^{n-m+1}
  \alpha_{p_1,\dots,p_m}
  \sum_{j_l=-p-1}^{q} \Phi^{(j_1,\dots,j_m)}(x)
 j_1^{p_1}\dots   j_m^{p_m}u^{(p_1)}(x)\dots
  u^{(p_m)}(x),\\
  &\Phi^{(j_1,\dots,j_m)}(x)=
  \frac{\partial^{m}\Phi}{\partial
    a_{j_1}\dots \partial a_{j_m}}(u(x),\dots,u(x)).
\end{aligned}
\end{equation}
For $n=1$, \eqref{eq:37} is satisfied with $\alpha_{p_1}=1$:
\begin{equation*}
  F'(0)=
  \Big(\sum_{j_1=-p-1}^{q} j_1  \frac{\partial\Phi}{\partial
    a_{j_1}}(u(x),\dots,u(x))
\Big)
  u'(x).
\end{equation*}
It follows from \eqref{eq:38} that $F'(0)=f'(u(x))u'(x)$ for any
sufficiently smooth $u$ if and only if
\begin{equation*}
\sum_{j=-p-1}^{q} j  \frac{\partial\Phi}{\partial
    a_{j}}(u(x),\dots,u(x))=f'(u(x)),
\end{equation*}
i.e.,
\begin{equation*}
\sum_{j=-p-1}^{q} j  \frac{\partial\Phi}{\partial
    a_{j}}(u,\dots,u)=f'(u).
\end{equation*}
Now, \eqref{eq:39} may be obtained by any means but, as we have just
seen, they are equivalent to \eqref{eq:38}, \eqref{eq:37}, which only depends on
derivatives of $u$ at $x$, as long as $u$ is smooth
enough at a convex open set containing the corresponding stencil.

Assume now that that $u(x-(p+1)h)$ is replaced by, $\widetilde
u(x-(p+1)h)$, the evaluation at
$x-(p+1)h$ of a polynomial that interpolates $u$ at another stencil
$x+p'h,\dots x+q'h$, with $p'\geq -p$. Let us denote
$\overline{q}=\max(q, q')$ and
\begin{align*}
  V(h)&=(\widetilde u(x-(p+1)h),u(x-ph),
\dots,u(x+qh)) = W(u(x-ph),\dots,u(x+\overline{q}h)),\\
G(h)&=\Phi(V(h)).
\end{align*}
If $u$ is
smooth in a convex open set containing $[x-(p+1)h,\dots,x+\overline{q}h]$,
then
\begin{equation*}
  u(x-(p+1)h)-\widetilde{u}(x-(p+1)h)=\bigO(h^s), s\geq 2
\end{equation*}
and, therefore,
\begin{align*}
  F(h)-G(h)=\frac{\partial{\Phi}(\xi_h)}{\partial
    a_{-p-1}}(u(x-(p+1)h)-\widetilde{u}(x-(p+1)h)) = \bigO(h^s).
\end{align*}
It  then follows from \eqref{eq:39} that
\begin{equation}\label{eq:40}
  G(h)=h f(u(x))_x + \bigO(h^{r'+1}), r'=\min(r, s-1),
\end{equation}
under the assumption of $u$ being smooth in a convex open set
containing $[x-(p+1)h,\dots,x+\overline{q}h]$. As argued above, this
is equivalent to
\begin{equation}\label{eq:41}
G'(0)=f(u(x))_x,\quad G^{(n)}(0)=0,\quad n=2,\dots,r',
\end{equation}
which translates into a functional relationship akin to that for
$F$. But now $G$ depends only on the stencil
$x-ph,\dots,x+\overline{q}h$ and therefore \eqref{eq:41} on  $u$ being
smooth in a convex open set containing $[x-ph,\dots,x+\overline{q}h]$,
i.e.,
\begin{equation*}
\frac{G(h)}{h}=f(u(x))_x+\bigO(h^{r'})
\end{equation*}
if $u$ is smooth in a convex open set containing
$[x-ph,\dots,x+\overline{q}h]$.

Therefore, we have concluded that $r$-th order extrapolation applied
to $r$-th order schemes decreases the order of the scheme near the
boundaries to $r-1$. It can be seen that this order is sharp. This is
in apparent contradiction with the results obtained in Section
\ref{ss:1dexperiments}, where fifth order extrapolation applied
together with fifth order schemes yields fifth order global errors,
both in $1$-norm and $\infty$-norm.

We next justify that the global errors are of order $r$  in $1$-norm
if the local error is of order $r+1$ at all but a bounded number (with
respect to $M$) of nodes where the order is a unit less.

Let us denote $v^n_j=u(x_j, t_n)$, $j=1,\dots,M$, for $M$ the total
number of spatial nodes,
 and
$\mathcal{H}=\mathcal{H}_{\lambda,h}$, $\lambda=\Delta t/h$,
the operator of the numerical scheme:
\begin{equation}\label{eq:ap1}
  u^{n+1}=\mathcal{H}u^n,
\quad \mathcal{H}:\mathbb{R}^M\rightarrow\mathbb{R}^M.
\end{equation}

With this notation, the global error is
$e^n=v^n-u^n$
and the local error can be written as
$f_j^n=v_{j}^{n+1}-\mathcal{H}v^n_j,$
which yields:
\begin{equation}\label{eq:ap2}
v^{n+1}=\mathcal{H}(v^n)+f^n.
\end{equation}
Subtracting \eqref{eq:ap1} and \eqref{eq:ap2} and using the mean value
theorem for $\mathcal{H}$ we obtain a relationship
between local and global error:
\begin{equation}\label{eleg}
e^{n+1}= B_ne^n+f^n,
\end{equation}
where
$$
B_n=\int_0^1\mathcal{H}'(u^n+s(v^n-u^n))ds,
$$
which,
by induction on $n$,  yields
\begin{equation}\label{eq:ap3}
e^n=\sum_{k=0}^{n-1}\left(\prod_{m=k+1}^{n-1}B_m\right)f^k.
\end{equation}

Now, we recall that the classical argument to obtain convergence from consistency and
stability proceeds as follows: Given  some fixed $T>0$,
if for $h \in (0, h_0)$,
\begin{equation}\label{eq:ap33}
  \Vert f^k\Vert \leq Ah^{r+1}, \quad \forall  k \text{ such that }
  k\Delta t \leq T,
\end{equation}
and, for
the induced matrix norm, $\Vert
B_m\Vert \leq 1+D h$, then it easily follows that $\Vert e^n\Vert \leq
C h^{r}$ for all $h\in (0, h_0)$ and all $n$ such that $n\Delta t=n\lambda h \leq T$.

This analysis is sufficient for the 1-norm if a bounded number of
nodes $x_j$ satisfy
$f_j^k=\bigO(h^r)$, whereas the rest
satisfy $f_j^k=\bigO(h^{r+1})$, since then $\Vert f_j
\Vert_1=h\sum_{j=1}^{M} |f_j^k| = \bigO(h^{r+1})$.
But this analysis does not explain that $\Vert e^{n}\Vert_{\infty} =
\bigO(h^r)$, since in this case  we have $\Vert f^k\Vert_{\infty}
=\bigO(h^{r})$, i.e., \eqref{eq:ap33} does not hold, so this argument
would yield $\Vert e^{n}\Vert_{\infty} =
\bigO(h^{r-1})$. A finer analysis, based
on the examination of the entries of the vectors involved in
\eqref{eq:ap3},
is then
needed to obtain $\Vert e^{n}\Vert_{\infty} = \bigO(h^r)$. The
analysis of
\begin{equation}\label{eq:ap4}
e_p^n=\sum_{k=0}^{n-1}\sum_{q=1}^{M}\Big(\prod_{m=k+1}^{n-1}B_m\Big)_{pq}f_q^k,
\end{equation}
can be quite involved in general and it is out of the scope of this
work. We will show a particular case in which we obtain global errors
of order 1 (in $\infty$-norm) from local errors with some bounded
number of entries of
order 1 and the rest of order 2. We consider the upwind scheme
\begin{equation*}
  u_{i}^{n+1}=(1-\lambda) u_{i}^{n} + \lambda u_{i-1}^{n}, i=1,\dots, M
\end{equation*}
for the linear advection equation $u_t+u_x=0$ with time dependent
boundary condition at $x=0$, $u(0, t)=g(t)$. The value of
$u_{0}^{n}\approx u(x_0, t_n)$
needs to be supplied according to this boundary condition. Prescribing
$u_{0}^{n}=g(t_n)=u(0, t_n)$ gives a first order approximation of $u(x_0,
t_n)=u(-h/2, t_n)$ and effectively yields a local error of order $r=1$ at
$x_1$:
\begin{equation*}
  |f_1^{n}|=|u(h/2, t_{n}+\Delta t)-(1-\Delta t/h)u(h/2, t_n)-\Delta
  t/h g(t_n)| \leq A_1h,
\end{equation*}
whereas the rest of the entries of the local error satisfy
\begin{equation*}
  |f_j^{n}| \leq Ah^2, j=2,\dots, M,
\end{equation*}
i.e., are of order $r+1=2$, for suitable positive constants $A_1, A$.

For this linear case, all the matrices $B_m$  are
$(1-\lambda)I_{M}+\lambda N_M$, where $N=N_M$ is
the nilpotent matrix  whose powers are given by
$(N^j)_{pq}=\delta_{p-q,j}$, $j\geq 1$, $ 1\leq p, q\leq M$.
We compute now, with the change of variables $k'=n-k-1 \to k$:
\begin{align*}
e_p^n&=\sum_{k=0}^{n-1}\sum_{q}^{M}((1-\lambda)I+\lambda N)^{n-k-1}_{pq}f_q^k
\\
&=\sum_{k=0}^{n-1}\sum_{q=1}^{M}\sum_{j=0}^{n-k-1} \binom {n-k-1}{j}
(1-\lambda)^{n-k-1-j} \lambda^{j}N^{j}_{pq}f_q^k\\
&=\sum_{k=0}^{n-1}\sum_{q=1}^{M}\sum_{j=0}^{n-k-1} \binom {n-k-1}{j}
(1-\lambda)^{n-k-1-j} \lambda^{j}\delta_{p-q,j}f_q^k\\
&=\sum_{k=0}^{n-1}\sum_{q=\max(1,p-n+k+1)}^{p}\binom {n-k-1}{p-q}
(1-\lambda)^{n-k-1-p+q} \lambda^{p-q}f_q^k\\
&=\sum_{k=0}^{n-1}\sum_{q=\max(1,p-k)}^{p}\binom {k}{p-q}
(1-\lambda)^{k-p+q} \lambda^{p-q}f_q^{n-k-1}\\
&=\sum_{q=1}^{p}
\Big(\lambda^{p-q}
\sum_{k=p-q}^{n-1}\binom {k}{p-q}
(1-\lambda)^{k-p+q}\Big) f_q^{n-k-1}.
\end{align*}
We can bound
\begin{equation*}
  \lambda^{s}
\sum_{k=s}^{n-1}\binom {k}{s}
(1-\lambda)^{k-s}
\leq   \lambda^{s}
\sum_{k=s}^{\infty}\binom {k}{s}
(1-\lambda)^{k-s}
\end{equation*}
for $s=p-q$ and use that
$\lambda^{s}
\sum_{k=s}^{n-1}\binom {k}{s}
(1-\lambda)^{k-s}
\leq   \lambda^{s}
\sum_{k=s}^{\infty}\binom {k}{s}
(1-\lambda)^{k-s}=\frac{1}{\lambda}$ (see next lemma) to conclude:
\begin{gather*}
|e_p^n|\leq \frac{1}{\lambda}\sum_{q=1}^{p} |f_q^{n-k-1}|\leq
\frac{1}{\lambda}(
A_1 h+(p-1) Ah^2) \leq
\frac{1}{\lambda}(
A_1 h+M Ah^2)=\frac{A_1+A}{\lambda}h,
\end{gather*}
for any $p=1,\dots,M$, i.e., $\Vert e^n\Vert_{\infty}=\bigO(h^r)$, $r=1$, as
proposed.

\begin{lemma}
  For any $\lambda \in(0,1)$ and $s \in\mathbb N$,
  \begin{equation*}
\sum_{k=s}^{\infty}\binom {k}{s}
(1-\lambda)^{k-s}=    \lambda^{-s-1}.
\end{equation*}

\end{lemma}
\begin{proof}
  Set $A_s=\sum_{k=s}^{\infty}\binom {k}{s}
(1-\lambda)^{k-s}$. From the identity
$\binom{k}{s}=\binom{k-1}{s}+\binom{k-1}{s-1}$, for $0\leq s < k$, and
the change of variables  $k'=k-1$ we get:
\begin{align*}
  \sum_{k=s}^{\infty}\binom {k}{s}
(1-\lambda)^{k-s} &=
1+\sum_{k=s+1}^{\infty}(\binom {k-1}{s}+\binom
{k-1}{s-1})(1-\lambda)^{k-s} \\
&=
1+\sum_{k=s+1}^{\infty}\binom
{k-1}{s}(1-\lambda)^{k-s}+\sum_{k=s+1}^{\infty}\binom
{k-1}{s-1}(1-\lambda)^{k-s}\\
&=
1+(1-\lambda)\sum_{k'=s}^{\infty}\binom
{k'}{s}(1-\lambda)^{k'-s}+\sum_{k'=s}^{\infty}\binom
{k'}{s-1}(1-\lambda)^{k'-s} \\
&=(1-\lambda)\sum_{k'=s}^{\infty}\binom
{k'}{s}(1-\lambda)^{k'-s}+\sum_{k'=s-1}^{\infty}\binom
{k'}{s-1}(1-\lambda)^{k'-s}.
\end{align*}
We obtain that $ A_{s}=(1-\lambda)A_s+A_{s-1}$ i.e.,
$A_{s}=\lambda^{-1}A_{s-1}$, which immediately yields
\begin{equation*}
A_{s}=\lambda^{-s}A_0=\lambda^{-s}\sum_{k=0}^{\infty}(1-\lambda)^k=\lambda^{-s-1}.
\end{equation*}

\end{proof}

\end{document}